\documentclass{article}[12pt]
\usepackage[title]{appendix}
\usepackage{lipsum}
\usepackage{amsfonts}
 \usepackage{algorithm,algorithmic}
\usepackage{graphicx}
\usepackage{epstopdf}
\usepackage{algorithmic}
\usepackage{mathrsfs}
\usepackage{mathtools}
\usepackage{color}
\usepackage{graphicx}
\usepackage{epstopdf}
\usepackage{stmaryrd}
 \usepackage{subfigure}
\usepackage{float}
\usepackage{amsmath}
    \usepackage{bm}
\usepackage{amsfonts,amssymb}
   \usepackage{natbib}
  \usepackage{dsfont}
\usepackage{amsfonts}
\usepackage{mathrsfs}
  \usepackage{pifont}
\numberwithin{equation}{section}
 \usepackage{amsthm}
 \newtheorem{theorem}{Theorem}[section]
\newtheorem{lemma}[theorem]{Lemma}

\newtheorem{remark}[theorem]{Remark}
\newtheorem{assumption}[theorem]{Assumption}
 \usepackage{amsmath}
 \usepackage[hidelinks]{hyperref}
 \usepackage{multirow}

\begin{document}
\title{
An Immersed Finite Element Method for  Anisotropic Elliptic Interface Problems with Nonhomogeneous Jump Conditions
}

\author{
Haifeng Ji\footnotemark[1]
\qquad\quad
Zhilin Li\footnotemark[2] 
}
\footnotetext[1]{School of Science, Nanjing University of Posts and Telecommunications, Nanjing 210023, China  (hfji@njupt.edu.cn)}
\footnotetext[2]{Department of Mathematics, North Carolina State University, Raleigh, NC 27695, USA  (zhilin@math.ncsu.edu)}

\date{}
\maketitle

\begin{abstract}
A new finite element method (FEM) using meshes that do not necessarily align with the interface is developed for two- and three-dimensional anisotropic elliptic interface problems with nonhomogeneous jump conditions. The degrees of freedom of the proposed method are the same as those of traditional nonconforming FEMs, while the function space is modified to account for the jump conditions of the solution. The modified function space on an interface element is shown to exist uniquely, independent of the element's shape and the manner in which the interface intersects it. Optimal error estimates for the method, along with the usual bound on the condition number of the stiffness matrix, are proven, with the error constant independent of the interface's  location relative to the mesh. To solve the resulting linear system, a preconditioner is proposed in which a Gauss-Seidel smoother with the interface correction is employed to ensure robustness against large jumps in the diffusion matrix. Numerical experiments are provided to demonstrate the optimal convergence of  the proposed method and the efficiency of the preconditioner.
\end{abstract}



\section{Introduction}\label{sec_introduce}
This paper concerns the numerical solution to anisotropic elliptic interface problems on unfitted meshes. 
Let $\Omega\subset \mathbb{R}^N$, $N\in\{2, 3\}$, be a convex polygonal (or polyhedral) domain, and let $\Gamma $  be a closed $C^2$-smooth ($N - 1$)-dimensional manifold  immersed in $\Omega$.  The interface $\Gamma$ divides $\Omega$ into two subdomains, $\Omega^+$ and $\Omega^-$, such that $\Gamma=\partial \Omega^-$.
We consider the following anisotropic elliptic interface problems with nonhomogeneous jump conditions,
\begin{subequations}\label{p1}
\begin{align}
-\nabla\cdot(\mathbb{B}\nabla u)&=f  \qquad \mbox{ in } \Omega^+\cup\Omega^-,\label{p1.1}\\
[u]_{\Gamma}&=g_D ~\quad \mbox{ on }  \Gamma,\label{p1.2}\\
[\mathbb{B}\nabla u\cdot n]_{\Gamma}&=g_N ~\quad \mbox{ on } \Gamma,\label{p1.3}\\
u&=0 \qquad \mbox{ on } \partial\Omega,\label{p1.4}
\end{align}
\end{subequations}
where $f\in L^2(\Omega)$, $g_D\in H^{3/2}(\Gamma)$ and $g_N\in H^{1/2}(\Gamma)$ are given functions, $n$ is the unit normal to $\Gamma$ pointing toward $\Omega^+$, and $[\cdot]_{\Gamma}$ stands for the jump across $\Gamma$, i.e., $[v]_{\Gamma}:=v^+|_{\Gamma}-v^-|_{\Gamma} $ with $v^\pm=v|_{\Omega^\pm}$.
The coefficient $\mathbb{B}(x)=(b_{ij}(x))$ is a symmetric positive definite  $N\times N$ matrix,  whose  components are piecewise smooth, but can be discontinuous across the interface $\Gamma$. Let $\mathbb{B}^\pm=\mathbb{B}|_{\Omega^\pm}$ and $b_{ij}^\pm=b_{ij}|_{\Omega^\pm}$. 
We assume that  $b_{ij}^\pm\in C^1(\overline{\Omega^\pm})$ and there exist positive constants $\beta^\pm_m$ and $\beta^\pm_M$ such that 
\begin{equation}\label{def_betaM}
\beta_{M}^\pm y^\top y\geq y^\top\mathbb{B}^\pm(x) y\geq \beta_{m}^\pm y^\top y, \quad \forall y \in\mathbb{R}^N  ~~\forall x\in\Omega^\pm.
\end{equation}

Equation (\ref{p1}) describes a common physical process in which diffusion is not only direction-dependent but also discontinuous. This problem has broad applications across various scientific fields, including materials science, crystal growth, Hele-Shaw flows, magnetic confinement fusion, and reservoir simulation, see for example, \cite{green2024efficient, chen2012effective, lu2020fully} and references therein.
Numerous numerical methods have been developed in the literature to address anisotropic problems or interface problems. For interface problems, considerable attention has focused on methods using unfitted meshes that are not necessarily aligned with the interface and enable  the use of a fixed background mesh regardless of interface configurations, enhancing computational efficiency and reducing both time and storage requirements when handling moving or complex interfaces, especially in three-dimensional spaces. 
There are numerous unfitted mesh methods in the literature, including Peskin's immersed boundary method \cite{Peskin1977Numerical}, the FEMs based on penalties \cite{babuvska1970finite,Barrett1987}, the immersed interface method (IIM) \cite{leveque1994immersed,dong2020fe}, the FEMs based on Nitsche’s method \cite{hansbo2002unfitted}, the extended FEM \cite{fries2010extended}, the cutFEM \cite{burman2018cut}, the immersed finite element (IFE) method \cite{li1998immersed,Li2003new}, and the immersed virtual element method \cite{Cao2021Immersed}, among others.

In this paper, we focus on the IFE method, which has a distinguishing feature compared to other unfitted FEMs, that is, its degrees of freedom are the same as those in traditional FEMs, while its basis functions are modified to account for the jump conditions of the solution. The IFE method can be regarded as a variant of the IIM, a finite difference method that modifies stencils at irregular points. This implies that the IFE space is isomorphic to the standard FE space on the same mesh, and thus the structure of the stiffness matrix is identical to that of traditional FEMs. This feature is particularly advantageous for designing efficient solvers and addressing moving interface problems \cite{guo2021SIAM}.
The first IFE method was proposed in \cite{li1998immersed} for solving one-dimensional interface problems and was recently revisited in \cite{wang2021robust}, where robust optimal error estimates with respect to the jump in discontinuous coefficients were derived, along with the development of an optimal multigrid algorithm. Extensions of the IFE method to high-dimensional problems, higher-degree formulations, and nonconforming cases, as well as related theoretical analyses, have been thoroughly explored; see, e.g., \cite{Li2003new,kafafy2005three,taolin2015siam,guo2020immersed, 201GUOSIAM,2021ji_IFE,ji2023immersed}.

However, relatively few works focus on three-dimensional IFE methods applied to problems with anisotropic coefficients and nonhomogeneous jump conditions (see \cite{hou2013weak, lu2020three}). To the best of our knowledge, no theoretical results are currently available for such methods. In this work, we propose an IFE method and provide a corresponding analysis for the interface problem.
For two-dimensional problems with nonhomogeneous jump conditions, several works exist, such as \cite{ygong-li, He2012Immersed, guzman2016higher, adjerid2023enriched,jo2024analysis}, in which various types of correction functions are constructed to handle the  nonhomogeneous jump conditions.
Given an interface element $T$, let $\Gamma_T=\Gamma\cap T$. The quantity  $\mathrm{avg}_{\Gamma_T}(g_N)=|\Gamma_T|^{-1}\int_{\Gamma_T}g_N$ is a natural choice for constructing the correction function on this element.
However, as demonstrated in \cite{zhang2023unfitted,ji2024mini},  $\mathrm{avg}_{\Gamma_T}(g_N)$ may become unbounded as $|\Gamma_T|\to 0$ when $g_N\in H^{1/2}(\Gamma)$.
Inspired by \cite{201GUOSIAM,zhuang2019high,ji2024mini}, we address this issue by using $\Gamma_T^{ext}$, an extension of $\Gamma_T$.
For the case with anisotropic coefficients, the author in \cite{ji2023immersed} shows that the Crouzeix-Raviart-type IFE basis functions, based on integral-value degrees of freedom, are unisolvent on arbitrary triangles or tetrahedra. In this work, we demonstrate that the correction functions are also unisolvent and provide a complete theoretical analysis.
It should be noted that IFE basis functions based on nodal-value degrees of freedom may not exist, even on isosceles right triangles, for anisotropic coefficients, as shown in \cite{An2014A}.  
In addition, we propose a preconditioner to efficiently solve the resulting linear system. To ensure robustness against large jumps in the diffusion coefficients, we employ a Gauss-Seidel smoother with the interface correction, inspired by  \cite{ludescher2020multigrid,chu2024multigrid}.

The rest of this paper is organized as follows. Section~\ref{sec_pre} reviews the notation and assumptions for unfitted meshes.
In Section~\ref{sec_finite}, the IFE method is presented, along with the explicit basis functions and correction functions. The theoretical analysis is provided in Section~\ref{sec_anal}.
Section~\ref{sec_cond} analyzes the condition number of the resulting linear system and introduces a preconditioner. 
Numerical examples are presented in Section~\ref{sec_num}. We conclude in the final section.

\section{Preliminaries}\label{sec_pre}
Let $k\geq 0$ be an integer and $1\leq p\leq \infty$.  We adopt the standard notation $W^k_p(D)$ for Sobolev spaces on a domain $D$, with the norm $\|\cdot\|_{W^k_p(D)}$ and the seminorm $|\cdot|_{W^k_p(D)}$. In particular, $W^k_2(D)$ is denoted by $H^{k}(D)$ with the norm  $\|\cdot\|_{H^{k}(D)}$ and the seminorm $|\cdot|_{H^{k}(D)}$. As usual, $H_0^1(D)=\{v\in H^1(D) : v=0 \mbox{ on }\partial D\}$.
For any subdomain $D\subset \mathbb{R}^N$, we define the subdomains  $D^\pm:=D\cap \Omega^\pm$ and the broken Sobolev spaces
\begin{equation*}
H^k(\cup D^\pm)=\{v\in L^2(D) : v|_{D^\pm} \in H^k(D^\pm)\},
\end{equation*}
equipped with the norm $\|\cdot\|_{H^k(\cup D^\pm)}$ and the  semi-norm $|\cdot|_{H^k(\cup D^\pm)}$, satisfying
$$
\|\cdot\|^2_{H^k(\cup D^\pm)}=\|\cdot\|^2_{H^k(D^+)}+\|\cdot\|^2_{H^k(D^-)}, \quad|\cdot|^2_{H^k(\cup D^\pm)}=|\cdot|^2_{H^k(D^+)}+|\cdot|^2_{H^k(D^-)}.
$$
Under the setting introduced in Section~\ref{sec_introduce}, problem (\ref{p1}) has a unique solution $u\in H^2(\cup\Omega^\pm)$ that satisfies the following a priori estimate (see \cite{bramble1996finite,McLean})
\begin{equation}\label{regular}
\|u\|_{H^2(\cup \Omega^\pm)}\leq C(\|f\|_{L^2(\Omega)}+\|g_N\|_{H^{1/2}(\Gamma)}+\|g_D\|_{H^{3/2}(\Gamma)}).
\end{equation}

For any $v\in H^m(\cup \Omega^\pm)$, $m>0$, let $v^\pm=v|_{\Omega^\pm}$. It is well known that there exist  extensions $v_E^\pm\in H^m(\mathbb{R}^N)$ such that (see \cite{Gilbargbook})
\begin{equation}\label{extension}
v_E^\pm|_{\Omega^\pm}=v^\pm~\mbox{and}~\| v_E^\pm\|_{H^m(\mathbb{R}^N)}\leq C\|v^\pm\|_{H^m(\Omega^\pm)}.
\end{equation}

Let $ \{\mathcal{T}_h\}_{h>0}$ be a family of shape-regular simplicial  triangulations of the domain $\Omega$, generated independently of the interface $\Gamma$. The mesh size of $\mathcal{T}_h$ is defined by $h=\max_{T\in\mathcal{T}_h}h_T$ with $h_T=\mathrm{diam}(T)$, where $\mathrm{diam}$ stands for the diameter. In this paper, the term ``face" refers to an edge in two dimensions and a face in three dimensions. 
We denote the set of faces of $\mathcal{T}_h$ by $\mathcal{F}_h$ and the set of faces of an element $T$ by $\mathcal{F}_T$.
The sets of interior and boundary faces are denoted  by $\mathcal{F}^\circ$ and  $\mathcal{F}^\partial$, respectively.
We adopt the convention that elements $T$ and faces $F$ are open sets. Then, the sets of interface elements and interface faces are defined by $\mathcal{T}_h^\Gamma =\{T\in\mathcal{T}_h :  T\cap \Gamma\not = \emptyset\}$ and $\mathcal{F}_h^\Gamma=\{ F\in \mathcal{F}_h : F \cap \Gamma\not = \emptyset\}$, respectively. The sets of non-interface elements and non-interface faces are denoted by $\mathcal{T}_h^{non}=\mathcal{T}_h\backslash\mathcal{T}_h^\Gamma$ and $\mathcal{F}_h^{non}=\mathcal{F}_h\backslash \mathcal{F}_h^\Gamma$, respectively.  For each $T\in\mathcal{T}_h$, we let $\Gamma_T=\Gamma\cap T$ and $T^\pm=T\cap \Omega^\pm$. It is evident that if $T\in\mathcal{T}_h^{non}$, then  $\Gamma_T=\emptyset$, and there exists an $s\in\{+,-\}$ such that $T^s=\emptyset$.


\begin{assumption}
 For each triangle $\triangle \in \mathcal{T}^\Gamma_h$ (if $N=2$) or  $\triangle \in \mathcal{F}^\Gamma_h$ (if $N=3$), the intersection $\Gamma\cap \partial \triangle$ consists of exactly two points, and these two points lie on different edges of the triangle $\triangle$.
 Additionally, for $N=3$, if $T \in \mathcal{T}^\Gamma_h$, then $\Gamma\cap \partial T$ consists of exactly one closed curve.
 \end{assumption} 
 
Since $\Gamma$ is $C^2$-smooth, the above assumption holds true if the mesh is fine enough. 
Elements containing faces that lie exactly on the interface can be classified as non-interface elements and handled by standard methods. For simplicity, we exclude these special cases from our discussion.
 
For each interface element $T\in\mathcal{T}_h^\Gamma$, 
we use a plane $L_T$ that is close to the interface so that the  piecewise IFE  basis functions can be constructed and analyzed.  Clearly, $L_T$ can be determined by a point $\bar{x}_{0,T}$ and a unit vector $\bar{n}_T$, such that $\bar{x}_{0,T}\in L_T$ and $\bar{n}_T \bot L_T$. We make the following assumption regarding the choice of $\bar{x}_{0,T}$ and $\bar{n}_T$:
\begin{assumption}\label{assum_x0t}
The plane $L_T$ and the interface $\Gamma$ intersect the same faces of $T$, and there exists a point $\tilde{x}_T\in\Gamma_T$ such that  
\begin{equation*}
\tilde{x}_T-\bar{x}_{0,T}  \perp L_T~  \mbox{ and }~|\bar{x}_{0,T}-\tilde{x}_T|+h_T|\bar{n}_T-n(\tilde{x}_T)|  \lesssim h_T^2,
\end{equation*}
where $|\cdot|$ denotes the 2-norm of a vector. 
\end{assumption}

Here and below, $\lesssim$ denotes less than or equal to up to a constant that is independent of the mesh size and the intersection configuration of the interface and the mesh. Similarly, $A \sim B$ means that $A\lesssim B$ and $ B\lesssim A$.

We can simply choose a point on $\Gamma_T$ as $\bar{x}_{0,T}$ and set $\bar{n}_T=n(\bar{x}_{0,T})$ to satisfy the assumption. Alternatively, an example that fulfills the assumption  can be obtained by using the signed distance function $d(x)|_{\Omega^\pm}=\pm\mathrm{dist}( x,\Gamma)$.
 Let $A_i$, $i=1,\cdots, N+1$ be the vertices of the element $T$. We choose $L_T=\{x :  d_{h,T}(x)=0\}$, where  $d_{h,T}\in\mathbb{P}_1(\mathbb{R}^N)$ satisfies $d_{h,T}(A_i)=d(A_i)$ for all $i=1,\cdots, N+1$. Let $\mathbf{p}_{L_T}$ be the orthogonal projection onto the plane $L_T$. The point $\bar{x}_{0,T}$  can then be chosen as $\bar{x}_{0,T}=\mathbf{p}_{L_T}(\tilde{x}_T)$, where $\tilde{x}_T\in \Gamma_T$ is an arbitrary point.

Let $\bar{t}_{i,T}$, $i=1,\cdots, N-1$, be standard basis vectors in the plane $L_T$ perpendicular to $\bar{n}_{T}$. We define a subregion of $L_T$ (a line segment for $N=2$ and a square for $N=3$) as follows:
\begin{equation*}
S_T=\{x\in L_T : x=\bar{x}_{0,T}+\sum_{i=1}^{N-1}\zeta_i\bar{t}_{i,T},~ 0\leq \zeta_i\leq \mu h_T \},
\end{equation*}
where $\mu>0$ is a fixed constant. Next, we define a mapping $\mathbf{p}_{h,T}: S_T\rightarrow \Gamma$  by $\mathbf{p}_{h,T}(x)=x+\varrho_h\bar{n}_T$, where  $\varrho_h$ is the smallest value such that $x+\varrho_h\bar{n}_T\in\Gamma$.  
\begin{assumption}\label{assum_phT}
The mapping $\mathbf{p}_{h,T}: S_T\rightarrow \Gamma$ is well-defined for each interface element $T$.
\end{assumption}

Define  $h_\Gamma=\max_{T\in\mathcal{T}_h^\Gamma}h_T$. The above assumption holds if $\mu h_\Gamma$ is sufficiently small  relative to the curvature of $\Gamma$. In practice,  $\mu$ is often chosen to be $0.5$.


To handle the jump condition (\ref{p1.2}), we introduce some points on $L_T$:
\begin{equation}\label{def_xit_bar}
\bar{x}_{i,T}=\bar{x}_{0,T}+ \mu h_T \bar{t}_{i,T}, \quad  i=1,\cdots, N-1,
\end{equation}
and the corresponding points on $\Gamma$: 
\begin{equation*}
\tilde{x}_{i,T}=\mathbf{p}_{h,T}(\bar{x}_{i,T}), \quad  i=0,\cdots, N-1.
\end{equation*}
Recalling Assumption~\ref{assum_x0t} and the fact that $\Gamma\in C^2$, we can use the tangent plane of $\Gamma$ at the point $\tilde{x}_T$ as a bridge to prove that 
\begin{equation}\label{est_x_t_x_star}
|\tilde{x}_{i,T}-\bar{x}_{i,T}|\lesssim h_T^2,\quad i=0,\cdots, N-1.
\end{equation}

To enforce the jump condition (\ref{p1.3}), we define the operator $\mathrm{avg}_{\Gamma_{T}^{ext}}$ by
\begin{equation*}
\mathrm{avg}_{\Gamma_{T}^{ext}}(v)=|\Gamma_{T}^{ext}|^{-1}\int_{\Gamma_{T}^{ext}}v,
\end{equation*}
where $\Gamma_T^{ext}=\{x :  x=\mathbf{p}_{h,T}(\bar{x}) ~ \forall \bar{x}\in S_T\}.$
Clearly, $|\Gamma_{T}^{ext}|\geq |S_T|  \gtrsim h_T^{N-1}$, where  $|\cdot|$ denotes the measure of the domain or manifold. Using  H\"older's inequality, we obtain 
\begin{equation}\label{ineq_ave}
|{\rm avg}_{\Gamma_{T}^{ext}}(v)|\leq |\Gamma_{T}^{ext}|^{-1/2}\|v\|_{L^2(\Gamma_{T}^{ext})}\lesssim h_T^{(1-N)/2}\|v\|_{L^2(\Gamma_{T}^{ext})}.
\end{equation}
For each interface element $T\in\mathcal{T}_h^\Gamma$, we introduce the smallest ball $\mathscr{B}_T$ such that $T$, $S_T$ and $\Gamma_{T}^{ext}$ are contained within $\mathscr{B}_T$. It is straightforward to verify that  $\mathrm{diam}(\mathscr{B}_T)\lesssim h_T$.
Let $\Gamma_{\mathscr{B}_T}=\Gamma\cap \mathscr{B}_T $. Using  (\ref{ineq_ave}) and the well-known trace inequality (see, e.g., \cite[Lemma 1]{guzman2018inf}),
\begin{equation}\label{tac_ine}
\|v\|_{L^2(\Gamma_{\mathscr{B}_T})}\lesssim (h_T^{-1/2}\|v\|_{L^2(\mathscr{B}_T)}+h_T^{1/2}|v|_{H^1(\mathscr{B}_T)})\quad \forall v\in H^1(\mathscr{B}_T),
\end{equation}
we obtain the following estimate:
\begin{equation}\label{trace_avg}
|{\rm avg}_{\Gamma_{T}^{ext}}(v)|\lesssim h_T^{-N/2}\|v\|_{L^2(\mathscr{B}_T)}+h_T^{1-N/2}|v|_{H^1(\mathscr{B}_T)}\quad \forall v\in H^1(\mathscr{B}_T).
\end{equation}

We will use the quantity $\mathrm{avg}_{\Gamma_{T}^{ext}}(g_N)$ to construct IFE functions (see (\ref{dis_jp_1})). It seems natural to use
$\mathrm{avg}_{\Gamma_{T}}(g_N)=|\Gamma_{T}|^{-1}\int_{\Gamma_{T}}g_N$ instead of  $\mathrm{avg}_{\Gamma_{T}^{ext}}(g_N)$.
However, since the interface $\Gamma$ can intersect the element $T$ in an arbitrary manner (i.e., $|\Gamma_T|$ may approach zero),  we observe that $\mathrm{avg}_{\Gamma_{T}}(g_N)$ could become unbounded. This issue arises from the low regularity of   $g_N\in H^{1/2}(\Gamma)$, as discussed in \cite[Example 3.1]{zhang2023unfitted}.

If $T\in\mathcal{T}_h^{non}$, we denote $\mathscr{B}_T=T$ for simplicity. We make the following assumption, which holds when the mesh is not strongly graded near the interface.
\begin{assumption}\label{assumption_finite}
There exists a constant $M$ such that for each $T \in  \mathcal{T}_h^\Gamma$, the number of elements in the set $\{T^\prime\in\mathcal{T}_h : \mathscr{B}_T\cap \mathscr{B}_{T^\prime}\not=\emptyset  \}$ is bounded by $M$.
\end{assumption}

Finally, for each interface element $T\in\mathcal{T}_h^\Gamma$, we define two constant matrices $\mathbb{B}_T^\pm=\mathbb{B}^\pm(\tilde{x}_{0,T})$. It is not hard to see that for all $x\in\Gamma_{T}^{ext}$,
\begin{equation}\label{esti_B}
|\bar{n}_T^\top\mathbb{B}_T^\pm-(n^\top\mathbb{B}^\pm)(x)|\leq |\bar{n}_T^\top(\mathbb{B}_T^\pm-\mathbb{B}^\pm(x))|+|(\bar{n}_T^\top-n^\top(x))\mathbb{B}^\pm(x)|\lesssim h_T.
\end{equation}

\section{Finite element approximation}\label{sec_finite}
\subsection{Local approximations}
Given an interface element $T\in\mathcal{T}_h^\Gamma$, we construct two linear functions $(\xi_T^+,\xi_T^-)\in \mathbb{P}_1(\mathbb{R}^N)\times\mathbb{P}_1(\mathbb{R}^N)$ to approximate $u_E^+$ and $u_E^-$ on this element, respectively. 
To achieve optimal approximations, $(\xi_T^+,\xi_T^-)$ should satisfy some discrete interface conditions according to  (\ref{p1.2}) and (\ref{p1.3}). To this end, we define linear functionals $\mathcal{J}_{i,T}: \mathbb{P}_1(\mathbb{R}^N)\times\mathbb{P}_1(\mathbb{R}^N)\rightarrow \mathbb{R}$ for all $i=0,\cdots,N$ by
\begin{equation}\label{def_J}
\mathcal{J}_{i,T}\left(~(\phi^+,\phi^-)^\top~\right)=\left\{
\begin{aligned}
&[\![\phi^\pm]\!](\bar{x}_{i,T}),\quad\quad\quad i=0,\cdots,N-1,   \\
&[\![\mathbb{B}_T^\pm\nabla \phi^\pm\cdot \bar{n}_T]\!],~\quad i=N.
\end{aligned}\right.
\end{equation} 
Here and below, we use the notation $[\![v^\pm]\!]:=v^+-v^-$ for any two functions $v^+$ and $v^-$.

Let  $L_T$ divide $T$ into two subelements $T_h^+$ and $T_h^-$, which approximate  $T^+$ and $T^-$, respectively. 
Given a domain $D\in\mathbb{R}^N$, let $\chi_{D}(x)$ be the characteristic function of $D$, i.e., $\chi_{D}$ is identically one  on $D$ and zero elsewhere. 
We then define linear functionals  $\mathcal{M}_{F,T}$ for all $F\in\mathcal{F}_T$ by
\begin{equation}\label{def_MFT}
\mathcal{M}_{F,T}\left(~(\phi^+,\phi^-)^\top~\right)=\frac{1}{|F|}\left(\int_{F} \chi_{T_h^+}\phi^++\int_{F} \chi_{T_h^-}\phi^-\right).
\end{equation}

From now on, we will  use the notation $\mathcal{J}_{i,T}(\phi^+,\phi^-)$ instead of $\mathcal{J}_{i,T}\left(~(\phi^+,\phi^-)^\top~\right)$, and $\mathcal{M}_{F,T}(\phi^+,\phi^-)$ instead of  $\mathcal{M}_{F,T}\left(~(\phi^+,\phi^-)^\top~\right)$ for simplicity. It should be  emphasized that  $\mathcal{M}_{F,T}$ and $\mathcal{J}_{i,T}$ are  linear functionals,  not  bilinear forms. For example, $\mathcal{M}_{F,T}(\phi_2^+,\phi_1^-)+\mathcal{M}_{F,T}(\phi_1^+,\phi_1^-)\not=\mathcal{M}_{F,T}(\phi_1^++\phi_2^+,\phi_1^-)$, and in fact, we have the identity $\mathcal{M}_{F,T}(\phi_2^+,\phi_1^-)+\mathcal{M}_{F,T}(\phi_1^+,\phi_1^-)=\mathcal{M}_{F,T}(\phi_1^++\phi_2^+,2\phi_1^-)$.

We state the following lemma, the proof of which is postponed to subsection~\ref{subsec_proof}.
\begin{lemma}\label{lem_unique}
For each interface element $T$ (a triangle for $N=2$ or a tetrahedron for $N=3$) without any angle restrictions,  the pair $(\phi^+,\phi^-)\in\mathbb{P}_1(\mathbb{R}^N)\times\mathbb{P}_1(\mathbb{R}^N)$ is uniquely determined by $\mathcal{J}_{i,T}$ for all $i=0, \cdots, N$ and $\mathcal{M}_{F,T}$ for all $F\in\mathcal{F}_T$.
\end{lemma}

According to the above lemma, there exists a unique pair $(\xi_T^+,\xi_T^-)\in \mathbb{P}_1(\mathbb{R}^N)\times\mathbb{P}_1(\mathbb{R}^N)$ such that 
\begin{subequations}\label{dis_jp}
\begin{align}
&\mathcal{J}_{i,T}(\xi_T^+,\xi_T^-)=\left\{
\begin{aligned}
&g_D(\tilde{x}_{i,T}),\quad \quad~ i=0,\cdots,N-1,\\
&\mathrm{avg}_{\Gamma_{T}^{ext}}(g_N),\quad i=N,
\end{aligned}\right. \label{dis_jp_1}\\
&\mathcal{M}_{F,T}(\xi_T^+,\xi_T^-)=\mathcal{M}_{F,T}(u_E^+,u_E^-), \quad \forall F\in\mathcal{F}_T.\label{dis_jp_2}
\end{align}
\end{subequations}

The following lemma establishes that $\xi_T^+$ and $\xi_T^-$ provide optimal approximations of  $u_E^+$ and $u_E^-$, respectively. 
\begin{lemma}\label{lem_xi_app}
For each $T\in\mathcal{T}_h^\Gamma$, the following estimate holds:
\begin{equation}\label{result_lem_xi}
|u_E^\pm-\xi_T^\pm|_{H^m(T)}\lesssim h_T^{2-m}\sum_{s=\pm}\|u_E^s\|_{H^2(\mathscr{B}_T)}, \quad m=0, 1.
\end{equation}
\end{lemma}
The proof of this lemma is provided in Appendix~\ref{sec_proof}.  

%
%
%
\subsection{Correction functions}
Clearly, $(\xi_T^+,\xi_T^-)$ can be decomposed as
\begin{equation}\label{xi_deco}
(\xi_T^+,\xi_T^-)=(\xi_T^{0,+},\xi_T^{0,-})+(\xi_T^{J,+},\xi_T^{J,-}),
\end{equation}
where $(\xi_T^{J,+},\xi_T^{J,-})\in \mathbb{P}_1(\mathbb{R}^N)\times\mathbb{P}_1(\mathbb{R}^N)$ satisfies 
\begin{equation}\label{def_corr}
\begin{aligned}
&\mathcal{J}_{i,T}(\xi_T^{J,+},\xi_T^{J,-})=\left\{
\begin{aligned}
&g_D(\tilde{x}_{i,T}), \quad~ i=0,\cdots,N-1,\\
&\mathrm{avg}_{\Gamma_{T}^{ext}}(g_N),\quad\quad i=N,
\end{aligned}\right.\\
&\mathcal{M}_{F,T}(\xi_T^{J,+},\xi_T^{J,-})=0, \quad \forall F\in\mathcal{F}_T,
\end{aligned}
\end{equation}
and $(\xi_T^{0,+},\xi_T^{0,-})\in \mathbb{P}_1(\mathbb{R}^N)\times\mathbb{P}_1(\mathbb{R}^N)$ satisfies 
\begin{equation}\label{def_xi_J}
\mathcal{J}_{i,T}(\xi_T^{0,+},\xi_T^{0,-})=0, \quad \mathcal{M}_{F,T}(\xi_T^{0,+},\xi_T^{0,-})=\mathcal{M}_{F,T}(u_E^+,u_E^-),
\end{equation}
 for all $i=0,\cdots, N$ and for all $F\in\mathcal{F}_T$.
Roughly speaking, $(\xi_T^{J,+},\xi_T^{J,-})$ captures the nonhomogeneous jumps of the exact solution, while $(\xi_T^{0,+},\xi_T^{0,-})$ handles the jump in the matrix coefficient $\mathbb{B}(x)$. 

Based on the above observations, we define the correction function  $u_h^J$ as follows:
\begin{equation}\label{def_corr1}
u_h^J|_{T^\pm}=\left\{
\begin{aligned}
&\xi_T^{J,\pm}\qquad \mbox{ if } T\in\mathcal{T}_h^\Gamma,\\
&0\qquad ~~~~\mbox{ if } T\in\mathcal{T}_h^{non}.
\end{aligned}\right.
\end{equation}
The existence and uniqueness of $u_h^J$ are guaranteed by Lemma~\ref{lem_unique}. We note that the correction function is precomputed using $g_D$ and $g_N$. An explicit formula for $u_h^J$ is provided in subsection~\ref{sec_exp_correc}.

\subsection{IFE spaces and interpolation operators}
According to (\ref{def_xi_J}), the local IFE space and its degrees of freedom are defined as follows.
On each interface element $T\in\mathcal{T}_h^{\Gamma}$, we define the local IFE space as
\begin{equation}\label{def_vhT}
V_h^{\mathrm{IFE}}(T)=\{ v : v|_{T^\pm}=v^\pm, v^\pm\in\mathbb{P}_1(\mathbb{R}^N),~ \mathcal{J}_{i,T}(v^+,v^-)=0  ~\forall i=0,\cdots,N\},
\end{equation}
and its degrees of freedom as
\begin{equation}\label{def_NFT}
\mathcal{N}_{F,T}^{\mathrm{IFE}}(v)=\mathcal{M}_{F,T}(v^+,v^-), \quad\forall  F\in\mathcal{F}_T.
\end{equation}
The associated IFE interpolation of $u$  is then defined as
\begin{equation*}
\Pi_T^\mathrm{IFE} u\in V_h^{\mathrm{IFE}}(T), \quad \mathcal{N}^\mathrm{IFE}_{F,T}(\Pi_T^\mathrm{IFE} u)=\mathcal{M}_{F,T}(u_E^+, u_E^-), \quad\forall  F\in\mathcal{F}_T.
\end{equation*}

In view of the above construction, it is not hard to see that $\xi_T^\pm=(\Pi_T^{\mathrm{IFE}} u+u_h^J|_T)^\pm$. As a consequence of Lemma~\ref{lem_xi_app}, we obtain the following result.
\begin{lemma}\label{lem_pi_t_esti}
For each $T\in\mathcal{T}_h^\Gamma$, the following estimate holds:
\begin{equation}
|u-(\Pi_T^{\mathrm{IFE}} u+u_h^J)|_{H^m(\cup T^\pm)}\lesssim h_T^{2-m}\sum_{s=\pm}\|u_E^s\|_{H^2( \mathscr{B}_T)},\quad m=0,1.
\end{equation}
\end{lemma}

On each non-interface element $T\in\mathcal{T}_h^{non}$, we use the standard Crouzeix--Raviart finite element, whose space and degrees of freedom are defined as follows (see \cite{crouzeix1973conforming}):
\begin{equation*}
V_h(T)=\mathbb{P}_1(T),\quad \mathcal{N}_{F}(v)=|F|^{-1}\int_Fv, \quad  F\in\mathcal{F}_T.
\end{equation*}
To keep the notation concise, for $T\in\mathcal{T}_h^{non}$, we set $V_h^{\mathrm{IFE}}(T)=V_h(T)$ and $\mathcal{N}_{F,T}^{\mathrm{IFE}}=\mathcal{N}_{F}$.
We then define the following global IFE spaces:
\begin{equation*}
\begin{aligned}
&V_h^{\mathrm{IFE}}=\{ v : v|_T\in V_h^{\mathrm{IFE}}(T)~ \forall T\in\mathcal{T}_h,~ \mathcal{N}^{\mathrm{IFE}}_{F,T_1^F}(v|_{T_1^F})=\mathcal{N}^{\mathrm{IFE}}_{F,T_2^F}(v|_{T_2^F}) ~ \forall F\in\mathcal{F}^\circ\},\\
&V_{h,0}^{\rm IFE}=\{v\in V_{h}^{\rm IFE} :  \int_F v=0~~ \forall F\in\mathcal{F}_h^\partial\},
\end{aligned}
\end{equation*}
where $T_1^F$ and $T_2^F$ are the elements sharing the face$F$. The global IFE interpolation operator $\Pi_h^\mathrm{IFE}: H^2(\cup\Omega^\pm) \rightarrow V_h^{\mathrm{IFE}} $ is defined as
\begin{equation}
(\Pi_h^\mathrm{IFE} u)|_T=\left\{
\begin{aligned}
&\Pi_T^\mathrm{IFE} u \quad~\mbox{ if } T\in\mathcal{T}_h^\Gamma, \\
&\Pi_T u \quad \quad\mbox{ if } T\in\mathcal{T}_h^{non}, 
\end{aligned}\right.
\end{equation}
where $\Pi_T: W(T)\rightarrow V_h(T)$ is the standard nonconforming FE interpolation operator defined by
\begin{equation}\label{def_st_pi}
\mathcal{N}_{F}(\Pi_T v)=\mathcal{N}_{F}(v)\quad \forall F\in\mathcal{F}_T.
\end{equation}
Here, $W(T):=\{v\in L^2(T) : v|_{F}\in L^2(F) ~\forall F\in \mathcal{F}_T\}$ is used to ensure that $\mathcal{N}_{F}(v)$ is well-defined.
Using Lemma~\ref{lem_pi_t_esti}, Assumption \ref{assumption_finite}, the standard interpolation error estimates on non-interface elements, and the extension stability (\ref{extension}),  one can easily derive
\begin{equation}\label{int_err_00}
\sum_{T\in\mathcal{T}_h}|u-(\Pi_h^{\mathrm{IFE}} u+u_h^J)|^2_{H^m(\cup T^\pm)}\lesssim h^{4-2m}\|u\|^2_{H^2(\cup\Omega^\pm)}, \quad m=0,1,
\end{equation}
which indicates that the space $V_h^{\mathrm{IFE}}+\{u_h^J\}$ has the optimal approximation capabilities. We emphasize that the hidden constant in the above inequality is independent of the interface location relative to the mesh, meaning that optimal approximation holds even when small cuts exist in the mesh.

\subsection{The IFE method}
We begin by introducing the notation for jumps and averages. 
For each face $F\in\mathcal{F}_h^\circ$, we associate a fixed unit normal vector $n_F$, and let $T_1^F$ and $T_2^F$ be the elements sharing $F$ as a face, with $n_F$ pointing from  $T^F_1$ to $T^F_2$. 
We then define the jump and average on $F$ as $[v]_F=v|_{T_1^F}-v|_{T_2^F}$ and $\{v\}_F=(v|_{T^F_1}+v|_{T^F_2})/2$, respectively. 
For each $F\in \mathcal{F}_h^\partial$, we set  $n_F$ to be the unit outward normal to $\partial \Omega$ and define $[v]_F=\{v\}_F=v$.   
Finally, on the interface $\Gamma$, the average is defined as $\{v\}_\Gamma=(v^++v^-)/2$.

The immersed finite method is formulated as follows:
 Find $u_h=u_h^{hom}+u_h^J$ with $u_h^{hom}\in V_{h,0}^{\rm IFE}$ such that  
\begin{equation}\label{method2}
A_h(u_h^{hom},v_h)=\int_{\Omega}fv_h-\int_\Gamma g_N\{v_h\}_\Gamma -A_h(u_h^J,v_h) ~~\forall v_h\in V_{h,0}^{\rm IFE}.
\end{equation}
The bilinear form $A_h(\cdot,\cdot)$ is given by
$$A_h(v,w)=a_h(v,w)+b_h(v,w)+s_h(v,w),$$
where
\begin{subequations}\label{def_AAH}
\begin{align}
&a_h(v,w)=\int_\Omega (\mathbb{B} \nabla_h v)\cdot\nabla_h w,\label{def_AAH_ah}\\
&b_h(v,w)=-\sum_{F\in\mathcal{F}_h^\Gamma}\int_F\left(\{(\mathbb{B}\nabla_h v)\cdot n_F\}_F[w]_F+\{(\mathbb{B}\nabla_h w)\cdot n_F\}_F[v]_F\right),\label{def_AAH_bh}\\
&s_h(v,w)=8\sum_{F\in\mathcal{F}_h^\Gamma}\int_{T_1^F\cup T_2^F}(\mathbb{B}\mathbf{r}_F([v]_F))\cdot \mathbf{r}_F([w]_F),\label{def_AAH_sh}
\end{align}
\end{subequations}
where $\nabla_h$ is a piecewise gradient defined by $(\nabla_h v)|_{T^\pm}=\nabla v|_{T^\pm}$ for all $T\in \mathcal{T}_h$, and $\mathbf{r}_F$ is the local  {\em  lifting operator} defined as follows.
For each interface face $F\in\mathcal{F}_h^\Gamma$, we first define the space 
\begin{equation*}
\mathbf{Q}_F=\{\mathbf{q}\in L^2(T^F_1\cup T^F_2)^N:~ \mathbf{q}|_{T_i^F}\in \nabla V^{\mathrm{IFE}}_h(T_i^F),~ i=1,2\},
\end{equation*}
where $\nabla V^{\mathrm{IFE}}_h(T)=\{ \nabla_h v : ~\forall v\in V^{\mathrm{IFE}}_h(T) \}$. We then define the local  {\em  lifting operator}  $\mathbf{r}_F : L^2(F)\rightarrow \mathbf{Q}_F$  by
\begin{equation}\label{def_lift}
\int_{T_1^F\cup T_2^F} (\mathbb{B} \mathbf{r}_F(v))\cdot \mathbf{q}=\int_Fv\{(\mathbb{B} \mathbf{q})\cdot n_F\}_F \qquad \forall \mathbf{q}\in \mathbf{Q}_F.
\end{equation}
%

\begin{remark}
We note that the plane $L_T$, an approximation of the interface $\Gamma$ on $T$, is used solely to provide connective conditions for the piecewise IFE functions.
The IFE space and the IFE method considered in this paper, however, are defined with respect to the exact interface $\Gamma$.
We do not consider the effects of numerical integration when calculating integrals over curved interfaces and subdomains for the following three reasons:
\begin{enumerate}
  \item It is common practice to approximate $\Gamma$ with a discrete interface $\Gamma_h$, which is continuous and piecewise linear but does not necessarily coincide with $L_T$ on each element $T$. 
Following similar arguments in \cite{bramble1996finite, ji2023immersed, ji2024mini}, we can show that for linear IFEs, using a piecewise linear approximation of the interface is sufficient to maintain optimal convergence.
  \item The study of IFE methods on curved subdomains is valuable for the development of higher-order IFE methods, for which a piecewise linear approximation of the interface may not be sufficient \cite{201GUOSIAM}.
  \item There are numerous highly accurate numerical quadrature methods available for integration on curved domains and interfaces, as discussed in the literature (e.g., \cite{cui2020high}). 
\end{enumerate}
\end{remark}

\begin{remark}
We also note that  the stabilization term $s_h(\cdot,\cdot)$  can be replaced by 
\begin{equation*}
\tilde{s}_h(v,w)=\sum_{F\in\mathcal{F}_h^\Gamma}\frac{\eta_F}{h_F} \int_F[v]_F[w]_F,
\end{equation*}
where $h_F$ denotes the diameter of $F$.
In this case, the parameter $\eta_F$ should be sufficiently large to ensure coercivity. Our analysis and theoretical results also hold for this case under the assumption that $\eta_F$ is large enough. 
\end{remark}

\subsection{Explicit IFE basis and correction functions}\label{subsec_expli}
To derive explicit expressions for the IFE basis and correction functions, we begin by proving Lemma~\ref{lem_unique}. 
Let $d_{L_T}$ be the signed distance function to the plane $L_T$,  defined as
\begin{equation}\label{def_d_LT}
d_{L_T}(x)=(x-\bar{x}_{0,T})\cdot \bar{n}_T.
\end{equation} 
In the following lemma, we prove a novel identity related to the function $\chi_{T_h^+}d_{L_T}$.
\begin{lemma}
For each interface element $T$, the following identity holds:
\begin{equation}\label{identi_T}
\nabla \Pi_T (\chi_{T_h^+}d_{L_T})=|T_h^+||T|^{-1}\bar{n}_T.
\end{equation}
\end{lemma}
\begin{proof}
Using the definition of $\Pi_T$ and the fact that $d_{L_T}=0$ on $L_T$, we have
\begin{equation*}
\int_{\partial T} \Pi_T(\chi_{T_h^+}d_{L_T})\nu=\int_{\partial T} \chi_{T_h^+}d_{L_T}\nu=\int_{\partial T_h^+} d_{L_T}\nu,
\end{equation*}
where $\nu$ denotes the unit outer normal along $\partial T$ and ${\partial T_h^+}$.
Applying the divergence theorem to both sides of the above identity, we obtain
\begin{equation*}
\int_T \nabla \Pi_T (\chi_{T_h^+}d_{L_T})=\int_{T_h^+} \nabla  d_{L_T}.
\end{equation*} 
Since $\nabla \Pi_T (\chi_{T_h^+}d_{L_T})$ is constant and $\nabla d_{L_T}= \bar{n}_T$, the proof is concluded.
\end{proof}


\subsubsection{Proof of Lemma~\ref{lem_unique}}\label{subsec_proof}
\begin{proof}
Since the dimension of $\mathbb{P}_1(\mathbb{R}^N)\times\mathbb{P}_1(\mathbb{R}^N)$ and the number of functionals $\mathcal{J}_{i,T}$ and $\mathcal{M}_{F,T}$ are equal, it suffices to prove that $\phi^+=\phi^-=0$ if $\mathcal{J}_{i,T}(\phi^+,\phi^-)=0$  for all $i=0,\cdots,N$ and $\mathcal{M}_{F,T}(\phi^+,\phi^-)=0$ for all $T\in\mathcal{F}_T$. From $\mathcal{J}_{i,T}(\phi^+,\phi^-)=0$ for all $i=0,\cdots,N-1$, we have $(\phi^+-\phi^-)|_{L_T}=0$.
Therefore, there exists some constant $\alpha$ such that
$\phi^+-\phi^-=\alpha d_{L_T}.$
Thus, we can write
\begin{equation}\label{pro_uni_01}
(\phi^+,\phi^-)=(\phi^-,\phi^-)+(\phi^+-\phi^-,0)=(\phi^-,\phi^-)+\alpha (d_{L_T}, 0).
\end{equation}
Next, using $\mathcal{M}_{F,T}(\phi^+,\phi^-)=0$, we obtain
$
\mathcal{M}_{F,T}(\phi^-,\phi^-)+\alpha \mathcal{M}_{F,T}(d_{L_T}, 0)=0,
$
which implies 
\begin{equation}\label{pro_uni_02}
\phi^-=-\alpha \Pi_T (\chi_{T_h^+}d_{L_T}).
\end{equation}
Now, using $\mathcal{J}_{N,T}(\phi^+,\phi^-)=0$ and the above identities, we can deduce
\begin{equation*}
\begin{aligned}
0=\mathcal{J}_{N,T}(\phi^-,\phi^-)+\alpha \mathcal{J}_{N,T}(d_{L_T}, 0)=\alpha\mathcal{J}_{N,T}(-\Pi_T (\chi_{T_h^+}d_{L_T})+d_{L_T},-\Pi_T (\chi_{T_h^+}d_{L_T})).
\end{aligned}
\end{equation*}
Using (\ref{def_J}), (\ref{identi_T}) and the fact that $\nabla d_{L_T}=\bar{n}_T$, we arrive at
\begin{equation}\label{eq_cT_uni}
\left( |T_h^+||T|^{-1}\bar{n}_T^\top(\mathbb{B}_T^--\mathbb{B}_T^+)\bar{n}_T+ \bar{n}_T^\top\mathbb{B}_T^+\bar{n}_T\right)\alpha=0.
\end{equation}
From (\ref{def_betaM}), we know that $\bar{n}_h^\top\mathbb{B}_T^+\bar{n}_h\geq \beta^+_m>0$, so we can divide by it to obtain
\begin{equation}\label{def_c_T}
 c_T\alpha=0 \mbox{ with } c_T:= |T_h^+||T|^{-1}\left(\frac{\bar{n}_T^\top\mathbb{B}_T^-\bar{n}_T}{\bar{n}_T^\top\mathbb{B}_T^+\bar{n}_T}-1\right)+ 1.
\end{equation}
Using (\ref{def_betaM}) and the fact that $ |T_h^+||T|^{-1}\in [0,1]$, we have the estimate 
\begin{equation}\label{esti_cT}
c_T\geq \mathrm{min}(1, \bar{n}_T^\top\mathbb{B}_T^-\bar{n}_T/(\bar{n}_T^\top\mathbb{B}_T^+\bar{n}_T))\geq \mathrm{min}(1, \beta_m^-/\beta_M^+)>0.
\end{equation}
Therefore, we must have $\alpha=0$. Combining this with (\ref{pro_uni_01}) and (\ref{pro_uni_02}), we conclude that $\phi^+=\phi^-=0$. This completes the proof of the lemma.
\end{proof}

\subsubsection{Explicit IFE basis functions}
By Lemma~\ref{lem_unique}, for each $F\in\mathcal{F}_T$, there exist functions $(\phi_{F,T}^+,\phi_{F,T}^-)\in\mathbb{P}_1(\mathbb{R}^N)\times\mathbb{P}_1(\mathbb{R}^N)$  such that 
\begin{equation}\label{def_psi_phi2}
\mathcal{J}_{j,T}(\phi_{F,T}^+,\phi_{F,T}^-)=0~~\forall j \quad\mbox{ and }\quad \mathcal{M}_{F^\prime,T}(\phi_{F,T}^+,\phi_{F,T}^-)=\delta_{FF^\prime}~~\forall F^\prime\in\mathcal{F}_T,
\end{equation}
where $\delta_{FF^\prime}$ is the Kronecker  delta, i.e., $\delta_{FF^\prime}=1$ if $F^\prime=F$, and $\delta_{FF^\prime}=0$ otherwise.

Let $\phi_{F,T}$ be the basis function of $V_h^{\mathrm{IFE}}(T)$ associated with $\mathcal{N}_{F,T}^{\mathrm{IFE}}$. Clearly, we have $\phi_{F,T}|_{T^\pm}=\phi_{F,T}^\pm$, where $\phi_{F,T}^\pm$ is defined in (\ref{def_psi_phi2}). The following lemma provides an explicit expression of $\phi_{F,T}^\pm$.

\begin{lemma}\label{lem_esti_basis}
Let $\lambda_{F,T}$ be the standard Crouzeix--Raviart basis function  on the element $T$ associated with the face $F\in\mathcal{F}_T$. Then, we have
\begin{equation}\label{expli_basis}
(\phi_{F,T}^+,\phi_{F,T}^-)=(\lambda_{F,T},\lambda_{F,T})+\alpha_T (  \phi_{T,J}^{+} ,  \phi_{T,J}^{-})
\end{equation}
with
\begin{subequations}\label{def_phi_alpha}
\begin{align}
& \phi_{T,J}^{+}:=d_{L_T}-\Pi_T (\chi_{T_h^+}d_{L_T}),\quad  \phi_{T,J}^{-}:=-\Pi_T (\chi_{T_h^+}d_{L_T}),\label{def_omega_pm}\\
&\alpha_T:=-[\![\mathbb{B}_T^\pm\nabla \lambda_{F,T}\cdot \bar{n}_T]\!]\left((\bar{n}_T^\top\mathbb{B}_T^+\bar{n}_T)c_T\right)^{-1},\label{alpha_T}
\end{align}
\end{subequations}
where the constant $c_T$ is defined in (\ref{def_c_T}).
Moreover, the basis functions satisfy the estimate
\begin{equation}\label{esti_basis}
|\phi_{F,T}^\pm|_{W^m_\infty(T)}\lesssim h_T^{-m},\quad m=0,1. 
\end{equation}
\end{lemma}
\begin{proof}
Using $\mathcal{J}_{i,T}(\phi^+_{F,T},\phi^-_{F,T})=0$ for all $i=0,\cdots,N-1$ from (\ref{def_psi_phi2}), we have $(\phi^+_{F,T}-\phi^-_{F,T})|_{L_T}=0$. Similarly to (\ref{pro_uni_01}), there  exists a constant $\alpha$ such that 
\begin{equation*}
(\phi^+_{F,T},\phi^-_{F,T})=(\phi^-_{F,T},\phi^-_{F,T})+\alpha (d_{L_T}, 0).
\end{equation*}
Using the fact that  $\mathcal{M}_{F^\prime,T}(\phi_{F,T}^+,\phi_{F,T}^-)=\delta_{FF^\prime}$ from (\ref{def_psi_phi2}), we get 
\begin{equation*}
\mathcal{M}_{F^\prime,T}(\phi^-_{F,T},\phi^-_{F,T})+\alpha \mathcal{M}_{F^\prime,T}(d_{L_T}, 0)=\delta_{FF^\prime} ~~\forall F^\prime\in \mathcal{F}_T,
\end{equation*}
which implies 
$
\phi^-_{F,T}=\lambda_{F,T}-\alpha \Pi_T (\chi_{T_h^+}d_{L_T}).
$
Therefore, we have
\begin{equation}\label{pro_exp_basis_0}
(\phi^+_{F,T},\phi^-_{F,T})=\alpha (~\overbrace{d_{L_T}-\Pi_T (\chi_{T_h^+}d_{L_T})}^{:= \phi_{T,J}^{+}},~ \overbrace{-\Pi_T (\chi_{T_h^+}d_{L_T})}^{:=\ \phi_{T,J}^{-}}~)+(\lambda_{F,T},\lambda_{F,T}).
\end{equation}
Similarly to (\ref{eq_cT_uni}), we use $\mathcal{J}_{N,T}(\phi^+_{F,T},\phi^-_{F,T})=0$, (\ref{def_J}), (\ref{identi_T}) and $\nabla d_{L_T}=\bar{n}_T$ to obtain an equation for $\alpha$,
\begin{equation*}
\begin{aligned}
\left(\bar{n}_T^\top\mathbb{B}_T^+\bar{n}_T\right)\underbrace{\left( |T_h^+||T|^{-1}\left(\frac{\bar{n}_T^\top\mathbb{B}_T^-\bar{n}_T}{\bar{n}_T^\top\mathbb{B}_T^+\bar{n}_T}-1\right)+ 1\right)}_{=c_T}\alpha=-[\![\mathbb{B}_T^\pm\nabla\lambda_{F,T}\cdot \bar{n}_T]\!],
\end{aligned}
\end{equation*}
where we note that $c_T$ is defined in (\ref{def_c_T}).
Substituting the solution, denoted by $\alpha_T$, into (\ref{pro_exp_basis_0}) yields the desired result (\ref{expli_basis}).  
Thanks to this explicit expression, the estimate (\ref{esti_basis}) can be derived easily by (\ref{esti_cT}), the definition $d_{L_T}(x)=(x-\bar{x}_{0,T})\cdot \bar{n}_T$ and the well-known estimate $|\lambda_{F,T}|_{W^m_\infty(T)}\lesssim h_T^{-m}$ for standard Crouzeix--Raviart basis functions.
\end{proof}

\begin{remark}\label{remark_IFEbasis}
According to (\ref{def_omega_pm}), we define $ \phi_{T,J}=\chi_{T^+}d_{L_T}-\Pi_T (\chi_{T_h^+}d_{L_T})$. Therefore, $\phi_{F,T}$ can be written as $\phi_{F,T}=\lambda_{F,T}+\alpha_T  \phi_{T,J}.$
\end{remark}

\subsubsection{Explicit correction functions}\label{sec_exp_correc}
By Lemma~\ref{lem_unique}, for all $i=0,\cdots, N$, there exist $(\psi_{i,T}^+,\psi_{i,T}^-)\in\mathbb{P}_1(\mathbb{R}^N)\times\mathbb{P}_1(\mathbb{R}^N)$ such that 
\begin{equation}\label{def_psi_phi1}
\mathcal{J}_{j,T}(\psi_{i,T}^+,\psi_{i,T}^-)=\delta_{ij}~~\forall j=0,\cdots, N,\quad \mathcal{M}_{F,T}(\psi_{i,T}^+,\psi_{i,T}^-)=0~~\forall F\in\mathcal{F}_T,
\end{equation}
where $\delta_{ij}$ is the Kronecker delta.
The  function in (\ref{def_corr}) can be written as 
\begin{equation}\label{explicit_correc}
\xi^{J,\pm}_T= \sum_{i=0}^{N-1}g_D(\tilde{x}_{i,T})\psi_{i,T}^\pm+\mathrm{avg}_{\Gamma_{T}^{ext}}(g_N)\psi_{N,T}^\pm.
\end{equation}
Recalling the definition of $u_h^J$ in (\ref{def_corr1}), it suffices to derive explicit expressions of  $\psi_{i,T}^\pm$.

Let $\omega_{i,T}^-=0$. We first construct $(\omega_{i,T}^+,\omega_{i,T}^-)\in\mathbb{P}_1(\mathbb{R}^N)\times\mathbb{P}_1(\mathbb{R}^N)$, which satisfies only the first condition of (\ref{def_psi_phi1}), i.e., $\mathcal{J}_{j,T}(\omega_{i,T}^+,\omega_{i,T}^-)=\delta_{ij}$ for all $j=0,\cdots, N$. By the definition of $\phi_{F,T}^\pm$ in (\ref{def_psi_phi2}), it can be verified that $\psi_{i,T}^\pm$  in (\ref{def_psi_phi1}) can be expressed as:
\begin{equation}\label{cons_psi}
\psi_{i,T}^\pm=\omega_{i,T}^\pm-\sum_{F\in\mathcal{F}_T}\mathcal{M}_{F,T}(\omega_{i,T}^+,\omega_{i,T}^-) \phi_{F,T}^\pm,\quad i=0,\cdots, N.
\end{equation}
From (\ref{def_J}) and (\ref{def_xit_bar}), we obtain the following for $i=0$:
\begin{equation}\label{ome_1_identi}
\begin{aligned}
&\omega_{0,T}^+(\bar{x}_{0,T})=1, ~~\nabla \omega_{0,T}^+ \cdot \bar{t}_{j,T}=-(\mu h_T)^{-1},\quad j=1,\cdots,N-1, \\
&\nabla \omega_{0,T}^+\cdot\bar{n}_T =-(\bar{n}_T^\top\mathbb{B}_T^+ \bar{n}_T)^{-1}\sum_{j=1}^{N-1}\bar{n}_T^\top\mathbb{B}_T^+\bar{t}_{j,T}  (\nabla \omega_{0,T}^+\cdot \bar{t}_{j,T}),
\end{aligned}
\end{equation}
where the relation $\nabla v= \bar{n}_T  (\nabla v \cdot\bar{n}_T)+\sum_{j=1}^{N-1} \bar{t}_{j,T}  (\nabla v \cdot \bar{t}_{j,T})$ for any function $v$ is used. This result allows us to explicitly express $\omega_{0,T}^+$. Similarly, for $i=1,\cdots, N-1$,  
\begin{equation}\label{ome_i2_identi}
\begin{aligned}
&\omega_{i,T}^+(\bar{x}_{0,T})=0, ~~\nabla \omega_{i,T}^+ \cdot \bar{t}_{j,T}=\left\{
\begin{aligned}
(\mu h_T)^{-1},\quad j=i,\\
0,\quad j\not=i, 
\end{aligned}\right.~~j=1,\cdots,N-1,
\\
&\nabla \omega_{i,T}^+\cdot\bar{n}_T =-(\bar{n}_T^\top\mathbb{B}_T^+ \bar{n}_T)^{-1}\sum_{j=1}^{N-1}\bar{n}_T^\top\mathbb{B}_T^+\bar{t}_{j,T}  (\nabla \omega_{i,T}^+\cdot \bar{t}_{j,T}),
\end{aligned}
\end{equation}
and for $i=N$,
\begin{equation}\label{ome_N_identi}
\omega_{N,T}^+(\bar{x}_{0,T})=\nabla \omega_{N,T}^+ \cdot \bar{t}_{j,T}=0,~j=1,\cdots,N-1, \nabla \omega_{N,T}^+\cdot\bar{n}_T =(\bar{n}_T^\top\mathbb{B}_T^+ \bar{n}_T)^{-1}.
\end{equation}

In addition to the explicit expression (\ref{explicit_correc}), we also present the following estimate, which will be used in the proof of Lemma~\ref{lem_xi_app}.

\begin{lemma}
For $m=0,1$, we have
\begin{equation}\label{esti_aux}
|\psi_{i,T}^\pm|_{W^m_\infty(T)}\lesssim\left\{
\begin{aligned}
&h_T^{-m}, ~\quad i=0,\cdots, N-1,\\
&h_T^{1-m}, \quad i=N.\\
\end{aligned}\right. 
\end{equation}
\end{lemma}
\begin{proof}
It follows from (\ref{ome_1_identi})--(\ref{ome_N_identi}) that, for $m=0,1$,
\begin{equation*}
|\omega_{i,T}^+|_{W^m_\infty}\lesssim\left\{
\begin{aligned}
&h_T^{-m}, ~\quad i=0,\cdots, N-1,\\
&h_T^{1-m}, \quad i=N.
\end{aligned}\right. 
\end{equation*}
Combining this with  $\omega_{i,T}^-=0$,  (\ref{cons_psi}), and (\ref{esti_basis}) yields the desired result.
\end{proof}

\section{Error estimates for the IFE method}\label{sec_anal}
\subsection{Coercivity and continuity}
We define the semi-norm $\|\cdot\|_h$ by $\|v\|_h^2=a_h(v,v)$. Clearly, $\|\cdot\|_h$ is indeed a norm on the IFE space $V_{h,0}^{\rm IFE}$ because $\|v\|_h=0$ implies $v$ is piecewise constant, and the zero boundary condition together with conditions on interfaces and faces implies $v\equiv 0$. 

The following lemma establishes that  $A_h(\cdot,\cdot)$  is coercive on the IFE space $V_{h,0}^{\rm IFE}$ with respect to the norm $\|\cdot\|_h$. The proof is similar to that in \cite[Lemma 5.1]{ji2023immersed}, and thus is omitted.
\begin{lemma}\label{lem_Coercivity}
We have
\begin{equation}\label{Coercivity}
A_h(v_h,v_h)\geq \frac{1}{2}\|v_h\|_{h}^2\qquad \forall v_h\in V_{h}^{\rm IFE}.
\end{equation}
\end{lemma}

Define an augmented norm $\interleave \cdot \interleave_h$ by
\begin{equation*}
\interleave v \interleave_h^2=\|v\|_{h}^2+\sum_{F\in\mathcal{F}_h^\Gamma}\left(h_F\|\{\mathbb{B}\nabla_h v\}_F\|^2_{L^2(F)}+h_F^{-1}\| [v]_F\|^2_{L^2(F)}\right)+s_h(v,v).
\end{equation*}
Using the Cauchy-Schwarz inequality, we obtain continuity of  $A_h(\cdot,\cdot)$,
\begin{equation}\label{conti}
|A_h(v,w)|\lesssim \interleave v \interleave_h \interleave w\interleave_h \quad\forall v,w \in  H^2(\cup\Omega^\pm)+ V_{h}^{\rm IFE}+\{u_h^J\}.
\end{equation}

\subsection{Norm-equivalence for IFE functions}
In this section, we demonstrate the equivalence of   the $\|\cdot\|_{h}$-norm and the $ \interleave\cdot\interleave_h$-norm on the IFE space $V_{h}^{\rm IFE}$.
To establish this, we first make the following assumption.
\begin{assumption}\label{assum_L1L2}
For any $T\in\mathcal{T}_h^\Gamma$, let $B_T$ be the largest ball inscribed in $T$, and let $L_1$ and $L_2$ be two parallel planes at smallest distance such that  $\Gamma_T$ is bounded by $L_1$ and $L_2$.
We assume that
$\mathrm{dist}(L_1,L_2)\leq \mathrm{diam}(B_T)/2$.
\end{assumption}

Since the interface $\Gamma$ is of class $C^2$, we have $\mathrm{dist}(L_1,L_2)\leq C_\Gamma h_T^2$ for some constant  $C_\Gamma>0$ depending only on $\Gamma$. 
Additionally, since the triangulation is shape-regular, there exists a constant  $\kappa>0$ such that $\kappa h_T\leq\mathrm{diam}(B_T)$. Therefore, the above assumption holds  if $h_\Gamma\leq \kappa/(2C_\Gamma)$.

\begin{lemma}\label{lem_vh_pm_vh}
For all  $T\in\mathcal{T}_h^\Gamma$ and all $v_h\in V_h^{\mathrm{IFE}}(T)$, 
\begin{equation}\label{vh_pm_vh01}
|  v_h^\pm|_{H^m(T)}\lesssim  |  v_h |_{H^m(\cup T^\pm)},\quad m=0,1.
\end{equation}
\end{lemma}
\begin{proof}
By (\ref{def_vhT}), it is straightforward to verify that $|\nabla v_h^+| \sim |\nabla v_h^-|$. The desired result (\ref{vh_pm_vh01}) for $m=0$ follows from the fact that
$$\|\nabla v_h^+\|_{L^2(T^-)}\sim \|\nabla v_h^-\|_{L^2(T^-)} \mbox{ and } \|\nabla v_h^-\|_{L^2(T^+)}\sim \|\nabla v_h^+\|_{L^2(T^+)}.$$

Next, we prove (\ref{vh_pm_vh01}) for $m=1$.
Note that $[\![v_h^\pm]\!]=v_h^+-v_h^-\in  \mathbb{P}_1(T)$ and $[\![v_h^\pm]\!]|_{L_T}=0$. We can write  $[\![v_h^\pm]\!]=d_{L_T} [\![\nabla v_h^\pm\cdot \bar{n}_T]\!]$, where $d_{L_T}$ is the signed distance function defined in (\ref{def_d_LT}). Since $\|d_{L_T}\|_{L^\infty(T)}\lesssim h_T$, we have 
\begin{equation}\label{pro_vh_pm1}
\|v_h^+\|_{L^2(T^-)}\lesssim \|v_h^-\|_{L^2(T^-)}+h_T |T^-|^{1/2}|[\![\nabla v_h^\pm]\!]|\lesssim \|v_h^-\|_{L^2(T^-)}+h_T^{1+N/2}|[\![\nabla v_h^\pm]\!]|.
\end{equation}
Note that both $|[\![\nabla v_h^\pm]\!]|\lesssim |\nabla v_h^+|$ and $|[\![\nabla v_h^\pm]\!]|\lesssim |\nabla v_h^-|$ hold. We now choose a superscript $s_0\in\{+,-\}$ such that there is a ball $\tilde{B}_T$ satisfying $\tilde{B}_T\subset T^{s_0}$ and 
$\mathrm{diam}(\tilde{B}_T)\gtrsim h_T$. The existence of such a ball $\tilde{B}_T$ is shown at the end of this proof. We then have
\begin{equation}\label{pro_vh_pm2}
|[\![\nabla v_h^\pm]\!]|\lesssim  |\nabla v_h^{s_0}|\lesssim  |\tilde{B}_T|^{-1/2}\|\nabla v_h^{s_0}\|_{L^2(\tilde{B}_T)}\lesssim h_T^{-N/2-1}\| v_h^{s_0}\|_{L^2(\tilde{B}_T)},
\end{equation}
where in the last inequality we use the standard  inverse inequality on $\tilde{B}_T$.
Combining (\ref{pro_vh_pm1}) and (\ref{pro_vh_pm2}), and using $\tilde{B}_T\subset T^{s_0}$, we obtain the desired result  (\ref{vh_pm_vh01}) for $m=1$.

Finally, we prove the existence of $\tilde{B}_T$ under Assumption~\ref{assum_L1L2}.
Let $B_T$ be the largest ball inscribed in $T$, and  let $L_1$ and $L_2$ be  two parallel planes bounding  $\Gamma_T$.
The planes $L_1$ and $L_2$ divide $B_T$ into subdomains $D_{0}$, $D_1$ and $D_2$ such that $D_1\subset T^+$ and $D_2\subset T^-$.
Let $B_i$ be the largest ball inscribed in $D_i$, $i=1,2$. Note that one of $D_1$ and $D_2$ may be empty. If $D_i=\emptyset$, we let $B_i=\emptyset$ and $\mathrm{diam}(B_i)=0$.
Now it is easy to see that 
$$\mathrm{dist}(L_1,L_2)+\mathrm{diam}(B_1)+\mathrm{diam}(B_2)\geq \mathrm{diam}(B_T)$$
Without loss of generality, assume $\mathrm{diam}(B_1)\geq \mathrm{diam}(B_2)$. We then choose $\tilde{B}_T=B_1$ and $s_0=+$. Hence, we have $\tilde{B}_T\subset T^{s_0}$. Using Assumption~\ref{assum_L1L2} we further have
\begin{equation*}
\begin{aligned}
2\mathrm{diam}(\tilde{B}_T)&\geq \mathrm{diam}(B_1)+\mathrm{diam}(B_2)\geq \mathrm{diam}(B_T)-\mathrm{dist}(L_1,L_2)\\
&\geq \mathrm{diam}(B_T)/2\gtrsim h_T.
\end{aligned}
\end{equation*}
This completes the proof of the lemma.
\end{proof}

A trace inequality for IFE functions is established in the following lemma.
\begin{lemma}\label{lema_trace}
For every  $T\in\mathcal{T}_h^\Gamma$, we have 
\begin{equation}\label{trace_IFE}
\|\nabla_h v_h\|_{L^2(\partial T)}\lesssim h_T^{-1/2}\|\nabla_h v_h\|_{L^2(T)}~\quad \forall v_h\in V_h^{\mathrm{IFE}}(T).
\end{equation}
\end{lemma}
\begin{proof}
Using the fact that $v_h^\pm \in \mathbb{P}_1(T)$ and (\ref{vh_pm_vh01}), we have
\begin{equation*}
\begin{aligned}
\|\nabla_h v_h\|^2_{L^2(\partial T)}&\leq \sum_{s=\pm}\|\nabla v_h^s\|^2_{L^2(\partial T)}\lesssim \sum_{s=\pm}h_T^{-1}\|\nabla v_h^s\|^2_{L^2(T)}\lesssim h_T^{-1}\|\nabla_h v_h\|^2_{L^2(T)}.
\end{aligned}
\end{equation*}
This completes the proof of the lemma.
\end{proof}

With the help of the above lemma, and similar to \cite[Lemma 5.3]{ji2023immersed}, we can derive the stability of the lifting $\mathbf{r}_F$. 
\begin{lemma}\label{lem_stab_lift}
For every  $F\in\mathcal{F}_h^\Gamma$, we have
\begin{equation*}
\|\mathbf{r}_F(v)\|_{L^2(T_1^F\cup T_2^F)}\lesssim h_F^{-1/2}\|v\|_{L^2(F)}\quad \forall v\in L^2(F).
\end{equation*}
\end{lemma}

We next consider the relation between the IFE function and its corresponding standard FE function.
For each $T\in\mathcal{T}_h^\Gamma$, we define $\widehat{\Pi}_T: V^{\mathrm{IFE}}_h(T)\rightarrow  V_h(T)$ by
\begin{equation}\label{def_Pi_hat}
\mathcal{N}_{F}(\widehat{\Pi}_T v_h)=\mathcal{N}_{F,T}^{\mathrm{IFE}}(v_h)\quad \forall F\in\mathcal{F}_T.
\end{equation}
For any $v_h\in V^{\mathrm{IFE}}_h(T)$, we define $\widehat{v}_h=\sum_{s=\pm}\chi_{T_h^s}v_h^s$, where $\chi_{T_h^s}$ is the characteristic function of $T_h^s$. By definition, we have $\widehat{v}_h\in C^0(T)$.
We note that $\widehat{\Pi}_Tv_h=\Pi_T\widehat{v}_h\not=\Pi_Tv_h$ due to the mismatch between $\Gamma $ and $L_T$ on $F$. 

\begin{lemma}\label{lem_hat_esti}
For every $T\in\mathcal{T}_h^\Gamma$, the following estimate holds:
\begin{equation}\label{hat_esti}
|v_h-\widehat{v}_h|_{H^m(\cup T^\pm)}\lesssim h_T^{5/2-2m}|v_h|_{H^1(\cup T^\pm)}~m=0,1,~\forall v_h\in V^{\mathrm{IFE}}_h(T).
\end{equation}
\end{lemma}
\begin{proof}
Let $T^\triangle = (T^+\backslash T_h^+)\cup (T^-\backslash T_h^-)$. By definition, we find 
\begin{equation*}
|v_h-\widehat{v}_h|_{H^m(\cup T^\pm)}=|[\![v_h^\pm]\!]|_{H^m(T^\triangle)}=|d_{L_T} [\![\nabla v_h^\pm\cdot \bar{n}_T]\!]|_{H^m(T^\triangle)}.
\end{equation*}
Using Assumption~\ref{assum_x0t} and (\ref{def_d_LT}), we know that $\|d_{L_T}\|_{W^m_\infty(T^\triangle)}\lesssim h_T^{2-2m}$ and $|T^\triangle|\lesssim h_T^{N+1}$. Therefore, 
\begin{equation*}
|v_h-\widehat{v}_h|_{H^m(\cup T^\pm)}\lesssim h^{N/2+5/2-2m}|[\![\nabla v_h^\pm]\!]| \lesssim h^{5/2-2m}\|[\![\nabla v_h^\pm]\!]\|_{L^2(T)},
\end{equation*}
which together with (\ref{vh_pm_vh01}) yields the desired result (\ref{hat_esti}).
\end{proof}

The following lemma establishes the norm equivalence between the IFE function and its corresponding FE function.
\begin{lemma}\label{lem_stand_IFE_eq}
For every $T\in\mathcal{T}_h^\Gamma$, we have
\begin{equation}\label{eq_stand_IFE}
| v_h |_{H^m(\cup T^\pm)}\sim | \widehat{\Pi}_T v_h  |_{H^m(T)},~~m=0,1, ~~\forall v_h\in V_h^{\mathrm{IFE}}(T).
\end{equation}
\end{lemma}
\begin{proof}
Using Remark~\ref{remark_IFEbasis}, (\ref{def_phi_alpha}) and  (\ref{def_Pi_hat}), we deduce the following identity for any $v_h\in V_h^{\mathrm{IFE}}(T)$:
\begin{equation}\label{pro_lem_stand_0}
v_h=\widehat{\Pi}_T v_h+\widehat{\alpha}  \phi_{T,J} ~~ \mbox{ with } ~~ \widehat{\alpha}:=-[\![\mathbb{B}_T^\pm\nabla (\widehat{\Pi}_T v_h)\cdot \bar{n}_T]\!]\left((\bar{n}_T^\top\mathbb{B}_T^+\bar{n}_T)c_T\right)^{-1}.
\end{equation}
From (\ref{esti_cT}), we have $|\widehat{\alpha}|\lesssim |\nabla \widehat{\Pi}_T v_h|$. Additionally, using (\ref{def_omega_pm}) and (\ref{def_d_LT}), we find 
$$|\phi_{T,J}|_{W^m_\infty(\cup T^\pm)}\lesssim h_T^{1-m}, \quad m=0,1.$$
Therefore,
\begin{equation}\label{pro_lem_stand_1}
|\widehat{\alpha}\phi_{T,J}|_{H^m(\cup T^\pm)}\lesssim h_T^{1-m}|\widehat{\Pi}_T v_h|_{H^1(T)}\lesssim |\widehat{\Pi}_T v_h|_{H^m(T)}, ~m=0,1,
\end{equation}
where we used the standard inverse inequality in the last inequality. Combining (\ref{pro_lem_stand_0}) and (\ref{pro_lem_stand_1}),  we obtain the first inequality of (\ref{eq_stand_IFE}):
\begin{equation*}
|v_h|_{H^m(\cup T^\pm)}\lesssim | \widehat{\Pi}_T v_h  |_{H^m(T)}.
\end{equation*}

To establish the reverse inequality, note that
\begin{equation*}
\widehat{\Pi}_T v_h=\sum_{F\in\mathcal{F}_T}\mathcal{N}_{F}(\widehat{\Pi}_T v_h)\lambda_{F,T}=\sum_{F\in\mathcal{F}_T}\mathcal{N}_{F,T}^{\mathrm{IFE}}(v_h)\lambda_{F,T},
\end{equation*}
where the last identity follows from (\ref{def_Pi_hat}).
Using $|\lambda_{F,T}|_{W^m_\infty(T)}\lesssim h_T^{-m}$, we get
\begin{equation*}
\begin{aligned}
| \widehat{\Pi}_T v_h  |_{H^m(T)}&\lesssim h_T^{N/2-m}\sum_{F\in\mathcal{F}_T}|\mathcal{N}_{F,T}^{\mathrm{IFE}}(v_h)|\\
&\lesssim  h_T^{N/2-m}\sum_{F\in\mathcal{F}_T} \sum_{s=\pm} |F|^{-1/2}\|v_h^s\|_{L^2(F)}.
\end{aligned}
\end{equation*}
Using $|F| \gtrsim h_T^{N-1}$, the standard  trace inequality, the standard inverse inequality, and Lemma~\ref{lem_vh_pm_vh}, we deduce 
\begin{equation*}
\begin{aligned}
&| \widehat{\Pi}_T v_h  |_{H^m(T)}\lesssim h_T^{1/2-m}\sum_{s=\pm} (h_T^{-1/2}\|v_h^s\|_{L^2(T)}+h_T^{1/2}|v_h^s|_{H^1(T)})\\
&\qquad \qquad\lesssim \sum_{s=\pm} (h_T^{-m}\|v_h^s\|_{L^2(T)}+|v_h^s|_{H^m(T)})\lesssim  h_T^{-m}\|v_h\|_{L^2(T)}+|v_h|_{H^m(\cup T^\pm)},
\end{aligned}
\end{equation*}
which completes the proof of (\ref{eq_stand_IFE}) for $m=0$. For $m=1$, applying the triangle inequality and  Lemma~\ref{lem_hat_esti}, we have
\begin{equation*}
\begin{aligned}
| \widehat{\Pi}_T v_h|_{H^1(T)}&=| \widehat{\Pi}_T v_h-c  |_{H^1(T)}\lesssim h_T^{-1}\|v_h-c\|_{L^2(T)}+|v_h|_{H^1(\cup T^\pm)}\\
&\lesssim h_T^{-1}\|\widehat{v}_h-c\|_{L^2(T)}+h_T^{-1}\|v_h-\widehat{v}_h\|_{L^2(T)}+|v_h|_{H^1(\cup T^\pm)}\\
&\lesssim |\widehat{v}_h|_{H^1(T)}+ |v_h|_{H^1(\cup T^\pm)}.
\end{aligned}
\end{equation*}
Finally, using the triangle inequality and Lemma~\ref{lem_hat_esti} again, we get
\begin{equation*}
| \widehat{\Pi}_T v_h|_{H^1(T)}\lesssim |v-\widehat{v}_h|_{H^1(\cup T^\pm)}+ |v_h|_{H^1(\cup T^\pm)}\lesssim |v_h|_{H^1(\cup T^\pm)}.
\end{equation*}
This completes the proof of the lemma.
\end{proof}

Using the above lemma, we can establish a bound on the jump of IFE functions across faces.
\begin{lemma}\label{lem_jum_less}
For every $F\in \mathcal{F}_h$, let $\mathcal{T}_h^F$ be the set of elements having $F$ as a face. Then we have
\begin{equation}\label{jum_less}
\|[v_h]_F\|_{L^2(F)}^2\lesssim h_F\sum_{T\in\mathcal{T}_h^F}|v_h|^2_{H^1(\cup T^\pm)} \quad  \forall v_h\in V_{h}^{\rm IFE}.
\end{equation}
\end{lemma}
\begin{proof}
If $F\in\mathcal{F}_h^{non}$ and $\mathcal{T}_h^F\subset\mathcal{T}_h^{non}$, (\ref{jum_less}) is a standard result. If $F\in\mathcal{F}_h^{non}$ and $\mathcal{T}_h^F\cap \mathcal{T}_h^\Gamma\not=\emptyset$, we can use Lemma~\ref{lem_vh_pm_vh} to get the desired result 
$$\|[v_h]_F\|_{L^2(F)}^2\lesssim h_F\sum_{T\in\mathcal{T}_h^F}|v_h^{s_0}|^2_{H^1(T)} \lesssim h_F\sum_{T\in\mathcal{T}_h^F}|v_h|^2_{H^1(\cup T^\pm)},$$
where the superscript $s_0=+$ or $-$ is chosen such that $F\subset \Omega^{s_0}$.
Now it remains to consider the case $F\in\mathcal{F}_h^\Gamma$ with $\mathcal{T}_h^F=\{T_1^F, T_2^F\}\subset \mathcal{T}_h^\Gamma$. For convenience, we let $F_i$, $i=0,\cdots,N$, be the faces of $T_1^F$ and let $F_i$, $i=0,-1,\cdots,-N$, be the faces of $T_2^F$. Clearly, $F=F_0$. Let $\phi_{F_i}$ be the global IFE basis function associated with $F_i$.
Using (\ref{def_Pi_hat}), we know that on $F$, 
\begin{equation*}
v_h=\sum_{i=-N}^Na_i\phi_{F_i},\quad a_i=\mathcal{N}_{F_i}(\widehat{\Pi}_T v_h) ~\mbox{ with } ~T=T_1^F \mbox{ or } T_2^F.
\end{equation*}
Therefore, we have
\begin{equation*}
\|[v_h]_F\|^2_{L^2(F)}=\|[v_h-a_0]_F\|^2_{L^2(F)}\lesssim \sum_{i=-N}^N (a_i-a_0)^2\|[\phi_{F_i}]_F\|^2_{L^2(F)}.
\end{equation*}
By (\ref{esti_basis}), we further have
\begin{equation*}
\|[v_h]_F\|^2_{L^2(F)}\lesssim \sum_{i=1}^N (a_i-a_0)^2h_{T_1^F}^{N-1}+\sum_{i=-N}^{-1} (a_i-a_0)^2h_{T_2^F}^{N-1}.
\end{equation*}
Using the standard result about the discrete $H^1$ semi-norm:
$$\sum_{i=1}^N (a_i-a_0)^2h_{T_1^F}^{N-2} \sim |\widehat{\Pi}_h v_h|^2_{H^1(T_1^F)},$$
we obtain 
\begin{equation*}
\|[v_h]_F\|^2_{L^2(F)} \lesssim h_F\sum_{T\in\mathcal{T}_h^F}|\widehat{\Pi}_h v_h|^2_{H^1(T)},
\end{equation*}
which together with (\ref{eq_stand_IFE}) yields the desired result (\ref{jum_less}).
\end{proof}

Using Lemmas \ref{lema_trace}, \ref{lem_stab_lift} and \ref{lem_jum_less}, 
we immediately derive the following norm equivalence on $V_{h}^{\rm IFE}$.
\begin{lemma}\label{lema_equ}
We have 
\begin{equation*}
\interleave v_h \interleave_h \sim \|v_h\|_{h}\quad \forall v_h\in V_{h}^{\rm IFE}.
\end{equation*}
\end{lemma}

\subsection{Consistency}
Using (\ref{p1}) and (\ref{method2}),  we have the error equation:
\begin{equation}\label{eqn_consis}
A_h(u-u_h,v_h)=\sum_{F\in\mathcal{F}_h^{non}}\int_F \mathbb{B}\nabla u\cdot n_F [v_h]_F-\int_\Gamma \mathbb{B}^-\nabla u^-\cdot n[v_h]_\Gamma:=(\mathrm{I})_1+(\mathrm{I})_2
\end{equation}
for all $v_h\in V_{h,0}^{\rm IFE}$

Similar to  (5.24) in \cite{ji2023immersed}, we obtain the following estimate for the first term:
\begin{equation}\label{est_I1}
|(\mathrm{I})_1| \lesssim h\sum_{s=\pm}\|u_E^s\|_{H^2(\Omega)}\|v_h\|_h \lesssim h\|u\|_{H^2(\cup\Omega^\pm)}\|v_h\|_h.
\end{equation}

For the second term $(\mathrm{I})_2$, we introduce the following lemma.
\begin{lemma}\label{lem_vh_gamma}
We have
\begin{equation*}
\|[v_h]_\Gamma\|^2_{L^2(\Gamma)}\lesssim h^{3}\sum_{T\in\mathcal{T}_h^\Gamma}|v_h|^2_{H^1(\cup T^\pm)} \quad \forall v_h\in V_{h}^{\rm IFE}.
\end{equation*}
\end{lemma}
\begin{proof}
On each element $T$,  since $[\![v_h^\pm]\!]$ is linear and $[\![v_h^\pm]\!]|_{L_T}=0$, we have
$$[v_h]_{\Gamma}(x)=[\![v_h^\pm]\!](x)=[\![\nabla v_h^\pm \cdot \bar{n}_T]\!]d_{L_T}(x)~~\forall x\in \Gamma_T.$$
 Similar to (\ref{est_x_t_x_star}), there holds $\|d_{L_T}\|_{L^\infty(\Gamma_T)}\lesssim h_T^2$. Therefore, 
\begin{equation*}
\begin{aligned}
\|[v_h]_\Gamma\|^2_{L^2(\Gamma)}&=\sum_{T\in\mathcal{T}_h^\Gamma}\int_{\Gamma_T}\left|[\![\nabla v_h^\pm \cdot \bar{n}_T]\!]d_{L_T}\right|^2\lesssim h^4 \sum_{T\in\mathcal{T}_h^\Gamma}\int_{\Gamma_T}\left|[\![\nabla v_h^\pm \cdot \bar{n}_T]\!]\right|^2\\
&\lesssim h^4\sum_{T\in\mathcal{T}_h^\Gamma}\sum_{s=\pm}\|\nabla v_h^s\|^2_{L^2(\Gamma_T)}\lesssim h^3\sum_{T\in\mathcal{T}_h^\Gamma}\sum_{s=\pm}\|\nabla v_h^s\|^2_{L^2(T)}\\
&\lesssim h^{3}\sum_{T\in\mathcal{T}_h^\Gamma}|v_h|^2_{H^1(\cup T^\pm)},
\end{aligned}
\end{equation*}
where we used  the trace inequality (\ref{tac_ine}) on $T$ and Lemma~\ref{lem_vh_pm_vh}. 
\end{proof}

It follows from the Cauchy-Schwarz inequality, the global trace inequality, and the above lemma that  
\begin{equation*}
|(\mathrm{I})_2|\lesssim \|\nabla u^-\|_{L^2(\Gamma)}\|[v_h]_\Gamma\|_{L^2(\Gamma)}\lesssim h^{3/2} \|u\|_{H^2{(\Omega^-)}}\|v_h\|_h.
\end{equation*}

Combining the results above, we obtain the following lemma:
\begin{lemma}\label{lem_consis}
Let $u$  and $u_h$ be the solutions of (\ref{p1}) and (\ref{method2}), respectively. Then, it holds that
\begin{equation*}
|A_h(u-u_h,v_h)|\lesssim h\|u\|_{H^2(\cup\Omega^\pm)}\|v_h\|_h \quad\forall v_h\in V_{h,0}^{\rm IFE}.
\end{equation*}
\end{lemma}

\subsection{A priori error estimates}
We are now ready to derive the $H^1$-error estimate for the IFE method.
\begin{theorem}\label{theo_h1}
Let $u$  and $u_h$ be the solutions of (\ref{p1}) and (\ref{method2}), respectively. 
Assume $h$ is sufficiently small. Then, the following estimate holds
\begin{equation*}
\interleave u-u_h \interleave_h\lesssim h\|u\|_{H^2(\cup \Omega^\pm)}.
\end{equation*}
\end{theorem}
\begin{proof}
Using the triangle inequality, we have 
\begin{equation}\label{pro_h1_01}
\interleave u-u_h \interleave_h\leq \interleave u-(\Pi_h^\mathrm{IFE} u+u_h^J) \interleave_h+\interleave (\Pi_h^\mathrm{IFE} u+u_h^J)-u_h \interleave_h
\end{equation}
Since $u_h=u_h^{hom}+u_h^J$, we can write 
\begin{equation*}
e_h^I:=(\Pi_h^\mathrm{IFE} u+u_h^J)-u_h=\Pi_h^\mathrm{IFE} u-u_h^{hom}\in V_{h,0}^{\mathrm{IFE}}.
\end{equation*}
It then follows from Lemmas~\ref{lema_equ} and \ref{lem_Coercivity} that 
\begin{equation*}
\begin{aligned}
\interleave e_h^I \interleave^2_h&\lesssim \|e_h^I\|^2_h\lesssim A_h((\Pi_h^\mathrm{IFE} u+u_h^J)-u_h,e_h^I)\\
&\lesssim A_h((\Pi_h^\mathrm{IFE} u+u_h^J)-u,e_h^I)+A_h(u-u_h,e_h^I),
\end{aligned}
\end{equation*}
which together with  (\ref{conti}),  Lemma~\ref{lem_consis} and  (\ref{pro_h1_01}) yields 
 \begin{equation*}
 \interleave u-u_h \interleave_h\lesssim \interleave u-(\Pi_h^\mathrm{IFE} u+u_h^J) \interleave_h+h\|u\|_{H^2(\cup\Omega^\pm)}.
\end{equation*}
It remains to estimate the first term on the right-hand side of the above inequality. On each interface element $T\in\mathcal{T}_h^\Gamma$, recall that $\xi_T^\pm=(\Pi_T^{\mathrm{IFE}} u+u_h^J|_T)^\pm$, where $\xi_T^\pm$ is defined by (\ref{dis_jp}). For simplicity of notation, let $\xi=\Pi_h^{\mathrm{IFE}} u+u_h^J$. By the standard trace inequality and Lemma~\ref{lem_xi_app}, we have
 \begin{equation*}
 \begin{aligned}
\sum_{F\in\mathcal{F}_h^\Gamma}h_F^{-1}\| [u-\xi]_F\|^2_{L^2(F)}&\leq \sum_{F\in\mathcal{F}_h^\Gamma}\sum_{s=\pm}h_F^{-1}\| [u_E^s-\xi^s]_F\|^2_{L^2(F)}\\
&\lesssim \sum_{s=\pm}\sum_{T\in\mathcal{T}_h^\Gamma}(h_T^{-2}\|u_E^s-\xi_T^s\|^2_{L^2(T)}+|u_E^s-\xi_T^s|^2_{H^1(T)})\\
&\lesssim h^2\sum_{s=\pm}\|u_E^s\|^2_{H^2(\Omega)}\lesssim h^2\|u\|^2_{H^2(\cup\Omega^\pm)}.
\end{aligned}
\end{equation*}
Using Lemma~\ref{lem_stab_lift}, we also have
\begin{equation*}
s_h(u-\xi,u-\xi)\lesssim h^2\|u\|^2_{H^2(\cup\Omega^\pm)}.
\end{equation*}
Similarly, we can show that
\begin{equation*}
\sum_{F\in\mathcal{F}_h^\Gamma}h_F\|\{\mathbb{B}\nabla_h (u-\xi)\}_F\|^2_{L^2(F)}\lesssim h^2\|u\|^2_{H^2(\cup\Omega^\pm)}.
\end{equation*}
Combining the above results with (\ref{int_err_00}) and using the definition of the $\interleave \cdot \interleave_h$-norm, we obtain
\begin{equation}\label{int_interleave_in}
\interleave u-(\Pi_h^\mathrm{IFE} u+u_h^J) \interleave_h\lesssim h\|u\|_{H^2(\cup\Omega^\pm)}.
\end{equation}
This completes the proof of the theorem.
\end{proof}

Next, we consider the $L^2$-error estimate. To achieve this, we utilize a $\delta$-strip argument to estimate functions in a tubular neighborhood of the interface.
Let $U_{\delta}(\Gamma)=\{ x\in\mathbb{R}^N : \mathrm{dist}( x,\Gamma)< \delta\}$ denote a tubular neighborhood of $\Gamma$ of thickness $\delta$. For sufficiently small $\delta$,  the following fundamental inequality holds \cite{huang2002mortar,Li2010Optimal}:
\begin{equation}\label{delta_strip}
\|v\|_{L^2(U_{\delta}(\Gamma))}\lesssim \sqrt{\delta}\|v\|_{H^1(U_{\delta_0}(\Gamma))},\quad \forall v\in H^1(U_{\delta_0}(\Gamma)),
\end{equation} 
where $\delta_0>0$ is a fixed constant.
\begin{theorem}
 Let $h$ be sufficiently small.  If $g_D=0$, we have
\begin{equation}\label{theo_l2_1}
\| u-u_h \|_{L^2(\Omega)} \lesssim h^2\|u\|_{H^2(\cup \Omega^\pm)}.
\end{equation}
If $g_D\not=0$, the error estimate becomes
\begin{equation}\label{theo_l2_2}
\| u-u_h \|_{L^2(\Omega)} \lesssim h^2(\|u\|_{H^2(\cup \Omega^\pm)}+\|u\|_{H^3(\cup U^\pm_{\delta_0}(\Gamma))}),
\end{equation}
where $\delta_0>0$ is a fixed constant and $U^\pm_{\delta_0}(\Gamma)=U_{\delta_0}(\Gamma)\cap \Omega^\pm$.
\end{theorem}
\begin{proof}
Let $e_h=u-u_h$. Consider the auxiliary problem
\begin{subequations}\label{auxi}
\begin{align}
&-\nabla\cdot(\mathbb{B}\nabla z)=e_h\qquad \qquad\mbox{ in }\Omega\backslash\Gamma,\\
&[z]_\Gamma=[\mathbb{B}\nabla z\cdot n]_\Gamma=0\qquad ~~\mbox{ on }\Gamma,\\
&z=0\qquad \qquad\qquad\qquad~~~~\mbox{ on }\partial \Omega.
\end{align}
\end{subequations}
Since $e_h\in L^2(\Omega)$ and $[z]_\Gamma=[\mathbb{B}\nabla z\cdot n]_\Gamma=0$, it follows from \eqref{regular} that
\begin{equation}\label{reg_L2}
z\in H^2(\cup \Omega^\pm)~~\mbox{ and }~~\|z\|_{H^2(\cup \Omega^\pm)}\lesssim \|e_h\|_{L^2(\Omega)}.
\end{equation}
Multiplying (\ref{auxi}) by $e_h$ and integrating by parts gives
\begin{equation*}
\begin{aligned}
\|e_h\|_{L^2(\Omega)}^2&=A_h(z,e_h)-\underbrace{\sum_{F\in\mathcal{F}_h^{non}}\int_F\mathbb{B}\nabla z\cdot n_F[e_h]_F}_{(\mathrm{II})_1}+\underbrace{\int_\Gamma \mathbb{B}^-\nabla z^-\cdot n [e_h]_\Gamma}_{(\mathrm{II})_2},
\end{aligned}
\end{equation*}
where we used $[z]_F=[\mathbb{B}\nabla z\cdot n_F]_F=s_h(z,u-u_h)=0$.
By (\ref{eqn_consis}) and the fact that $\Pi_h^{\mathrm{IFE}}z\in V_{h,0}^{\mathrm{IFE}}$,
\begin{equation*}
\begin{aligned}
&A_h(z,e_h)=A_h(z-\Pi_h^{\mathrm{IFE}}z,e_h)-A_h(\Pi_h^{\mathrm{IFE}}z,e_h)\\
&=\underbrace{A_h(z-\Pi_h^{\mathrm{IFE}}z,e_h)}_{(\mathrm{II})_3}-\underbrace{\sum_{F\in\mathcal{F}_h^{non}}\int_F \mathbb{B}\nabla u\cdot n_F [\Pi_h^{\mathrm{IFE}}z]_F}_{(\mathrm{II})_4}+\underbrace{\int_\Gamma \mathbb{B}^-\nabla u^-\cdot n[\Pi_h^{\mathrm{IFE}}z]_\Gamma}_{(\mathrm{II})_5}.
\end{aligned}
\end{equation*}
Combining these results, we obtain
\begin{equation}\label{pro_l2_00}
\|e_h\|_{L^2(\Omega)}^2\leq |(\mathrm{II})_1|+|(\mathrm{II})_2|+|(\mathrm{II})_3|+|(\mathrm{II})_4|+|(\mathrm{II})_5|.
\end{equation}

\textbf{Estimation of  term $(\mathrm{II})_3$}: 
Using (\ref{conti}), (\ref{int_interleave_in}) and Theorem~\ref{theo_h1}, we have
\begin{equation}\label{pro_l2_01}
|(\mathrm{II})_3|\leq \interleave z-\Pi_h^{\mathrm{IFE}}z \interleave_h \interleave e_h\interleave_h\lesssim h^2\|z\|_{H^2(\cup \Omega^\pm)}\|u\|_{H^2(\cup\Omega^\pm)},
\end{equation}
where we used $z_h^J=0$ (defined similarly to $u_h^J$) because $[z]_\Gamma=[\mathbb{B}\nabla z\cdot n]_\Gamma=0$.

\textbf{Estimation of  terms $(\mathrm{II})_1$ and $(\mathrm{II})_4$}: 
Using standard results from nonconforming FE analysis (see \cite[Chapter 10.3]{brenner2008mathematical}) and following a similar approach as in the estimation of (\ref{est_I1}), we derive
\begin{equation}\label{pro_l2_02}
|(\mathrm{II})_1|+|(\mathrm{II})_4|\lesssim h^2\|z\|_{H^2(\cup \Omega^\pm)}\|u\|_{H^2(\cup\Omega^\pm)}.
\end{equation}

\textbf{Estimation of  term $(\mathrm{II})_5$}: 
For all $T\in\mathcal{T}_h^\Gamma$, it holds that $T\subset U_{\delta}(\Gamma)$ with $\delta \lesssim h$. Using the Cauchy-Schwarz inequality, the global trace inequality,  Lemma~\ref{lem_vh_gamma}, and the $\delta$-strip argument (\ref{delta_strip}) with $\delta \lesssim h$, we obtain
\begin{equation}\label{pro_l2_03}
\begin{aligned}
|(\mathrm{II})_5|&\lesssim h^{3/2}\|u\|_{H^2(\Omega^-)}\|\nabla_h \Pi_h^{\mathrm{IFE}}z\|_{L^2(U_{\delta}(\Gamma))}\\
&\lesssim h^{3/2}\|u\|_{H^2(\Omega^-)}(\|\nabla_h (\Pi_h^{\mathrm{IFE}}z-z)\|_{L^2(\Omega)}+ \|\nabla_h z \|_{L^2(U_{\delta}(\Gamma))})\\
&\lesssim h^2\|u\|_{H^2(\Omega^-)}\|z\|_{H^2(\cup \Omega^\pm)}.
\end{aligned}
\end{equation}

\textbf{Estimation of  term $(\mathrm{II})_2$}: 
We first split the term into two parts: 
\begin{equation}\label{pro_l2_04}
\begin{aligned}
|(\mathrm{II})_2|\leq \left|\int_\Gamma \mathbb{B}^-\nabla z^-\cdot n [\Pi_h^{\mathrm{IFE}} u-u_h^{hom}]_\Gamma\right|+\left|\int_\Gamma \mathbb{B}^-\nabla z^-\cdot n [u-\Pi_h^{\mathrm{IFE}} u-u_h^J]_\Gamma\right|.
\end{aligned}
\end{equation}
By Lemma~\ref{lem_vh_gamma}, Theorem ~\ref{theo_h1} and (\ref{int_interleave_in}), the first term can be estimated as 
\begin{equation}\label{pro_l2_05}
\begin{aligned}
&\left|\int_\Gamma \mathbb{B}^-\nabla z^-\cdot n [\Pi_h^{\mathrm{IFE}} u-u_h^{hom}]_\Gamma\right|\lesssim h^{3/2}\|z\|_{H^2(\Omega^-)}\|\Pi_h^{\mathrm{IFE}} u-u_h^{hom}\|_h\\
&\qquad\qquad\lesssim h^{3/2}\|z\|_{H^2(\Omega^-)}\left\|(\Pi_h^{\mathrm{IFE}} u+u_h^J-u)+(u-u_h)\right\|_h\\
&\qquad\qquad\lesssim h^{5/2}\|z\|_{H^2(\Omega^-)}\|u\|_{H^2(\cup\Omega^\pm)}.
\end{aligned}
\end{equation}
It remains to estimate the second term.
For each interface element $T\in\mathcal{T}_h^\Gamma$, recall $\xi_T^\pm=(\Pi_T^{\mathrm{IFE}} u+u_h^J|_T)^\pm$, where $\xi_T^\pm$ is defined by (\ref{dis_jp}).

\textbf{Case 1}: $[u]_\Gamma=g_D=0$.  We have $[\![\xi_T^\pm]\!]|_{L_T}=0$ for all $T\in\mathcal{T}_h^\Gamma$, and thus, 
$[\Pi_h^{\mathrm{IFE}} u+u_h^J]_{\Gamma}(x)=[\![\nabla \xi_T^\pm]\!]d_{L_T}(x)$ for all $x\in\Gamma_T$.
Therefore, similar to the proof of Lemma~\ref{lem_vh_gamma}, we have
\begin{equation*}
\|[\Pi_h^{\mathrm{IFE}} u+u_h^J]_\Gamma\|^2_{L^2(\Gamma)}\lesssim h^3\sum_{T\in\mathcal{T}_h^\Gamma}\left|[\![\xi_T^\pm]\!] \right|^2_{H^1(T)}.
\end{equation*}
Now the second term can be estimated similarly to (\ref{pro_l2_03}), 
\begin{equation}\label{pro_l2_06}
\begin{aligned}
&\left|\int_\Gamma \mathbb{B}^-\nabla z^-\cdot n [u-\Pi_h^{\mathrm{IFE}} u-u_h^J]_\Gamma\right|=\left|\int_\Gamma \mathbb{B}^-\nabla z^-\cdot n [\Pi_h^{\mathrm{IFE}} u+u_h^J]_\Gamma\right|\\
&\qquad \lesssim h^{3/2}\|z\|_{H^2(\Omega^-)}\left(\sum_{T\in\mathcal{T}_h^\Gamma}\left|[\![\xi_T^\pm]\!] \right|^2_{H^1(T)}\right)^{1/2}\\
&\qquad \lesssim h^{3/2}\|z\|_{H^2(\Omega^-)}\left(\left(\sum_{T\in\mathcal{T}_h^\Gamma}\left|[\![\xi_T^\pm-u_E^\pm]\!] \right|^2_{H^1(T)}\right)^{1/2}+\left|[\![u_E^\pm ]\!]\right|_{H^1(U_{\delta}(\Gamma))}\right)\\
&\qquad \lesssim h^2\|z\|_{H^2(\Omega^-)}\|u\|_{H^2(\cup\Omega^\pm)},
\end{aligned}
\end{equation}
where we used Lemma~\ref{lem_xi_app} and the $\delta$-strip argument (\ref{delta_strip}) with $\delta \lesssim h$ in the last inequality. Combining (\ref{reg_L2})--(\ref{pro_l2_06}) yields the desired result (\ref{theo_l2_1}).

\textbf{Case 2}: $[u]_\Gamma=g_D\not = 0$. Using the trace inequality (\ref{tac_ine}) on $T$, Lemma~\ref{lem_xi_app}, and the $\delta$-strip argument (\ref{delta_strip}) with $\delta \lesssim h$, we get
\begin{equation}\label{pro_l2_07}
\begin{aligned}
&\left|\int_\Gamma \mathbb{B}^-\nabla z^-\cdot n [u-\Pi_h^{\mathrm{IFE}} u-u_h^J]_\Gamma\right|\lesssim \|z\|_{H^2(\Omega^-)}\|[u-\Pi_h^{\mathrm{IFE}} u-u_h^J]_\Gamma\|_{L^2(\Gamma)}\\
&\qquad \lesssim \|z\|_{H^2(\Omega^-)}\sum_{s=\pm}\left(\sum_{T\in\mathcal{T}_h^\Gamma}h_T^{-1}\|u^s_E-\xi_T^s\|^2_{L^2(T)}+h_T|u^s_E-\xi_T^s|^2_{H^1(T)}\right)^{1/2}\\
&\qquad \lesssim h^{3/2}\|z\|_{H^2(\Omega^-)}\sum_{s=\pm}\|u_E^s\|_{H^2(U_{\delta}(\Gamma))} \lesssim h^{2}\|z\|_{H^2(\Omega^-)}\sum_{s=\pm}\|u_E^s\|_{H^3(U_{\delta_0}(\Gamma))}\\
&\qquad \lesssim h^{2}\|z\|_{H^2(\Omega^-)}\|u\|_{H^3(\cup U^\pm_{\delta_0}(\Gamma))},
\end{aligned}
\end{equation}
which together with (\ref{reg_L2})--(\ref{pro_l2_05}) yields (\ref{theo_l2_2}). This completes the proof.
\end{proof}

\begin{remark}
The estimate (\ref{theo_l2_2}) for the case $g_D\not=0$ requires the exact solution $u$ to be slightly smoother near the interface.
A similar phenomenon occurs in the analysis of traditional FEMs when handling nonhomogeneous Dirichlet boundary conditions using nodal interpolation of the boundary data (see \cite{fix1983finite}).
\end{remark}

\section{Condition number analysis and preconditioners}\label{sec_cond}
Let $\{\phi_i\}_{i=1}^{\mathrm{ndof}}$ be the basis of  $V_{h,0}^{\rm IFE}$ and  ${\rm \mathbf{A}}$ be the stiffness matrix with entries 
${\rm \mathbf{A}}_{ij}=A_h(\phi_i,\phi_j)$ for all  $i,j=1,\cdots,\mathrm{ndof}$. 
Denote the standard nonconforming FE space by $V_{h,0}:=\{v : v|_T\in V_h(T) ~\forall T\in\mathcal{T}_h,~   \int_F [v]_F=0~ \forall F\in\mathcal{F}_h\}$, and define the operator $\widehat{\Pi}_h: V^{\mathrm{IFE}}_{h,0}\rightarrow  V_{h,0}$ by $(\widehat{\Pi}_hv)|_T=v|_T$ if $T\in\mathcal{T}_h^{non}$ and $(\widehat{\Pi}_hv)|_T=\widehat{\Pi}_Tv$ if $T\in\mathcal{T}_h^{\Gamma}$, where $\widehat{\Pi}_T$ is defined in (\ref{def_Pi_hat}). The stiffness matrix of the standard nonconforming FEM, denoted by ${\rm \mathbf{A}^{std}}$, can be computed by  ${\rm \mathbf{A}}_{ij}^{\rm std}=a_h(\widehat{\Pi}_h\phi_i,\widehat{\Pi}_h\phi_j)$. Suppose the family of triangulations is also quasi-uniform.
Then,  the $l_2$ condition number of the matrix satisfies  ${\rm cond}_2({\rm \mathbf{A}^{std}})\lesssim h^{-2}$.

Let $v_h=\sum_{i} \mathbf{v}(i)\phi_i$ with a vector $\mathbf{v}\in\mathbb{R}^{\mathrm{ndof}}$. Thus, $\widehat{\Pi}_hv_h=\sum_{i} \mathbf{v}(i)\widehat{\Pi}_h\phi_i$. It follows from (\ref{conti}) and Lemmas~\ref{lem_Coercivity}, \ref{lem_stand_IFE_eq} and \ref{lema_equ} that  
\begin{equation*}
\mathbf{v}^\top{\rm \mathbf{A}}\mathbf{v}=A_h(v_h,v_h)\sim \|v_h\|_h \sim \|\widehat{\Pi}_hv_h\|_h \sim a_h(\widehat{\Pi}_hv_h,\widehat{\Pi}_hv_h)=\mathbf{v}^\top{\rm \mathbf{A}^{std}}\mathbf{v}
\end{equation*}
and
\begin{equation*}
\|v_h\|^2_{L^2(\Omega)} \sim \|\widehat{\Pi}_hv_h\|^2_{L^2(\Omega)} \sim \mathbf{v}^\top\mathbf{v}.
\end{equation*}
Therefore, we have 
\begin{equation*}
{\rm cond}_2({\rm \mathbf{A}}) \sim {\rm cond}_2({\rm \mathbf{A}^{std}})\lesssim h^{-2},
\end{equation*}
which indicates that  the condition number of the stiffness matrix of the IFE method has the usual bound $O(h^{-2})$ with the hidden constant independent of the  interface location. 

Based on the above observation, we can conclude that  ${\rm \mathbf{A}}$ is  spectrally equivalent to ${\rm \mathbf{A}^{std}}$, and thus the preconditioners of the traditional nonconforming FEM can also be applied to the IFE method. However, the hidden constant in the above estimates depends on the contrast of the coefficient $\mathbb{B}^\pm$. If the contrast is small, using ${\rm \mathbf{A}^{std}}$ as a preconditioner works well. For high-contrast problems, however, specialized preconditioners are required to maintain efficiency and robustness. 

We first introduce the subdomain and the set of faces near the interface (see \cite{chu2024multigrid}):
\begin{equation*}
\Omega_{h}^{\Gamma}=\bigcup \{\overline{T} : \overline{T} \cap \overline{T}^\prime\not=\emptyset, T^\prime\in \mathcal{T}_h^{\Gamma},  T\in \mathcal{T}_h\},~~\mathcal{F}_h^{\Gamma, near}=\{F\in\mathcal{F}_h: F\subset \mathrm{int}(\Omega_{h}^{\Gamma}) \}.
\end{equation*}
Assuming the number of faces in $\mathcal{F}_h^{\Gamma, near}$ is $m_\Gamma$ and that the global indices of these faces are assigned at the beginning, we define ${\rm \mathbf{A}}^\Gamma$ with entries ${\rm \mathbf{A}}_{ij}^\Gamma = A_h(\phi_i, \phi_j)$ for all $i, j = 1, \dots, m_\Gamma$.  Let ${\rm \mathbf{I}}_l=(0,0,\cdots,1,\cdots,0)^\top\in \mathbb{R}^{\mathrm{ndof}}$ be a vector with a $1$ in the $l$-th position and ${\rm \mathbf{I}}_\Gamma=({\rm \mathbf{I}}_1,{\rm \mathbf{I}}_2,\cdots,{\rm \mathbf{I}}_{m_\Gamma},\mathbf{0},\cdots,\mathbf{0})$ be an ${\mathrm{ndof}}\times{\mathrm{ndof}}$ matrix.
The smoother matrix ${\rm \mathbf{R}}$ is defined in Algorithm \ref{smooth_ic} (see \cite{ludescher2020multigrid,chu2024multigrid}).
\begin{algorithm}
\caption{Gauss–Seidel smoother with interface correction}\label{smooth_ic}
\begin{algorithmic}[!h]
\STATE For any $\mathbf{g}\in\mathbb{R}^{\mathrm{ndof}}$, define ${\rm \mathbf{R}}\mathbf{g}$ as follows:
\STATE (1)  Set $\mathbf{v}_0=0$;
\STATE (2)  For $l=1,\cdots, \mathrm{ndof}$, define $\mathbf{v}_l=\mathbf{v}_{l-1}+{\rm \mathbf{I}}_l{\rm \mathbf{A}}_{ll}^{-1}{\rm \mathbf{I}}_l^\top(\mathbf{g}-{\rm \mathbf{A}}\mathbf{v}_{l-1})$;
\STATE (3)  For $l=\mathrm{ndof}+1$, define $\mathbf{v}_l=\mathbf{v}_{l-1}+{\rm \mathbf{I}}_\Gamma({\rm \mathbf{A}}^{\Gamma})^{-1}{\rm \mathbf{I}}_\Gamma^\top(\mathbf{g}-{\rm \mathbf{A}}\mathbf{v}_{l-1})$;
\STATE (4) Set ${\rm \mathbf{R}}\mathbf{g}=\mathbf{v}_{\mathrm{ndof}+1}$.
\end{algorithmic}
\end{algorithm}

We define the preconditioner for the stiffness matrix ${\rm \mathbf{A}}$ in Algorithm~\ref{preconditioner}.
\begin{algorithm}[!h]
\caption{Preconditioner matrix ${\rm \mathbf{B}}$ for the stiffness matrix ${\rm \mathbf{A}}$}\label{preconditioner}
\begin{algorithmic}
\STATE For any $\mathbf{g}\in\mathbb{R}^{\mathrm{ndof}}$, define ${\rm \mathbf{B}}\mathbf{g}$ as follows:
\STATE (1)  Set $\mathbf{v}_0=0$;
\STATE (2)  For $l=1,\cdots, n_s$, define $\mathbf{v}_l=\mathbf{v}_{l-1}+{\rm \mathbf{R}}(\mathbf{g}-{\rm \mathbf{A}}\mathbf{v}_{l-1})$;
\STATE (3)  For $l=n_s+1$, define $\mathbf{v}_l=\mathbf{v}_{l-1}+({\rm \mathbf{A}^{std}})^{-1}(\mathbf{g}-{\rm \mathbf{A}}\mathbf{v}_{l-1})$;
\STATE (4) For $l=n_s+2,\cdots, 2n_s+1$, define $\mathbf{v}_l=\mathbf{v}_{l-1}+{\rm \mathbf{R}}^\top(\mathbf{g}-{\rm \mathbf{A}}\mathbf{v}_{l-1})$;
\STATE (5) Set ${\rm \mathbf{B}}\mathbf{g}=\mathbf{v}_{2n_s+1}$.
\end{algorithmic}
\end{algorithm}

The inverse of ${\rm \mathbf{A}}^{\Gamma}$ is computed using a sparse direct solver based on sparse Cholesky factorization since it is of small size. The inverse of ${\rm \mathbf{A}^{std}}$, on the other hand, is computed using the Preconditioned Conjugate Gradient (PCG) method with the standard V-cycle multigrid algorithm as a preconditioner (see \cite{xu1992iterative}). To ensure insensitivity to high-contrast coefficients, the Gauss–Seidel smoother with interface correction is also employed in the standard V-cycle multigrid algorithm.

\section{Numerical experiments}\label{sec_num}
In this section, we present numerical experiments to demonstrate the performance of the proposed method. 
The resulting linear system is solved using the PCG method with the preconditioner defined in Algorithm~\ref{preconditioner}, where $n_s=1$ is set. The stopping criterion is that the relative residual error is less than $10^{-8}$, and the number of PCG iterations is denoted by ${\rm Iter_1}$ . 
For computing  $({\rm \mathbf{A}^{std}})^{-1}\mathbf{r}$ using the PCG method with a multigrid  preconditioner. The number of smoothing iterations in the V-cycle multigrid algorithm is set to $5$, and the stopping criterion is that either the absolute residual error or the relative residual error is less than $10^{-8}$, as the residual $\mathbf{r}$ may be very small during the algorithm. The maximum number of PCG iterations for different residuals $\mathbf{r}$ during the algorithm is denoted by  ${\rm Iter_2}$ .

To facilitate the construction of the exact solution, we replace (\ref{p1.4}) with the nonhomogeneous Dirichlet boundary condition $u|_{\partial \Omega}=g_{\partial \Omega}$. We examine the convergence rate of IFE solutions using the norms
$\|e_h\|_{L^2}:=\|u-u_h\|_{L^2(\Omega)}$ and   $|e_h|_{H^1}:=|\nabla_h(u-u_h)|_{L^2(\Omega)}$.

\textbf{Example 1}. We first test a two-dimensional example. We choose $\Omega=(-1,1)^2$ and 
 $\Gamma= \{(x,y) : x^2+y^2=0.5^2 \}$. The exact solution and the discontinuous coefficient are chosen as follows:

\begin{equation*}
\begin{aligned}
u^+(x,y)&=\ln (x^2+y^2),~ u^-(x,y)=\sin (x+y),\\
\mathbb{B}^+(x,y)&=
\left(
\begin{array}{ll}
\sin(x+y)+5& \cos (x+y)+2\\
\cos (x+y)+2& \sin(x+y)+10
\end{array}
\right)\beta_0^+,\\
\mathbb{B}^-(x,y)&=
\left(
\begin{array}{ll}
x^2+10& xy+2\\
xy+2& x^2y^2+5
\end{array}
\right)\beta_0^-,
\end{aligned}
\end{equation*}
where $\beta_0^+$ and $\beta_0^-$ are positive constants. The functions $f, g_N, g_D$ and $g_{\partial \Omega}$ are then determined by the above functions. We use uniform triangulations consisting of $2M^2$ congruent triangles. Numerical results reported in Tables~\ref{ex1_1}-\ref{ex1_3}  clearly show that the optimal convergence rates and the numbers of iterations are almost insensitive to the mesh size and the contrast of the coefficients.

\begin{table}[H]
\caption{Numerical results of  \textbf{Example 1} with $\beta_0^+=10^3$ and $\beta_0^-=1$. ${\rm Iter_1}$ denotes the number of PCG iterations for solving ${\rm \mathbf{AU}=\mathbf{F}}$. ${\rm Iter_2}$ denotes the maximum number of PCG iterations for solving ${\rm \mathbf{A}^{std}x}=\mathbf{r}$. The tolerance is set to be $10^{-8}$.\label{ex1_1}}
\begin{center}
\begin{tabular}{|c|c|c |c|}
\hline
     $M$  &   $\|e_h\|_{L^2}$ ~~(rate)  &    $|e_h|_{H^1} $~~(rate)        &      ${\rm Iter_1(Iter_2)}$      \\ \hline
        16   &   3.736E-02      (~---~)   &   6.806E-01     (~---~)  &      4       (--)  \\ \hline
        32   &   8.981E-03      (2.06)   &   3.538E-01      (0.94)    &     5         (8)     \\ \hline
        64   &   2.252E-03      (2.00)    &  1.701E-01      (1.06)    &     5         (9) \\ \hline
       128  &    5.393E-04     (2.06)   &   9.572E-02      (0.83)   &      5        (11) \\ \hline
       256  &    1.307E-04      (2.05)   &   4.566E-02      (1.07)   &      5        (13) \\ \hline
       512  &    3.229E-05      (2.02)   &   2.134E-02      (1.10)   &      5        (14) \\ \hline
      1024 &     7.956E-06      (2.02)  &    1.055E-02      (1.02)  &       5       (14)       \\ \hline
\end{tabular}
\end{center}
\end{table}

\begin{table}[H]
\caption{Numerical results of  \textbf{Example 1} with $\beta_0^+=1$ and $\beta_0^-=10^3$.\label{ex1_2}}
\begin{center}
\begin{tabular}{|c|c|c |c|}
\hline
     $M$  &   $\|e_h\|_{L^2}$ ~~(rate)  &    $|e_h|_{H^1} $ ~~(rate)        &      ${\rm Iter_1(Iter_2)}$       \\ \hline
       16   &   2.879E-02      (~---~)  &          6.076E-01    (~---~)     &    4         (--)  \\ \hline
        32  &    7.542E-03      (1.93)     & 3.161E-01      (0.94)    &     5         (7)  \\ \hline
        64  &    1.886E-03      (2.00)     & 1.586E-01      (1.00)    &    5        (10)  \\ \hline
       128 &     4.864E-04     (1.96)    &  8.014E-02      (0.98)   &    5        (11)  \\ \hline
       256 &     1.229E-04     (1.98)    &  4.020E-02      (1.00)   &    5        (12)  \\ \hline
       512 &     3.116E-05      (1.98)    &  2.016E-02      (1.00)   &     5        (14)  \\ \hline
      1024&      7.893E-06     (1.98)   &   1.010E-02      (1.00)  &      5        (15)  \\ \hline
\end{tabular}
\end{center}
\end{table}

\begin{table}[H]
\caption{Numerical results of  \textbf{Example 1} with $\beta_0^+=2$ and $\beta_0^-=1$.\label{ex1_3}}
\begin{center}
\begin{tabular}{|c|c|c |c|}
\hline
     $M$  &   $\|e_h\|_{L^2}$ ~~(rate)  &    $|e_h|_{H^1} $~~ (rate)        &      ${\rm Iter_1(Iter_2)}$       \\ \hline
        16  &    3.092E-02     (~---~)  &    6.166E-01    (~---~)   &      3         (--)\\ \hline
        32   &   7.950E-03      (1.96)    &  3.156E-01      (0.97)      &   3         (7)\\ \hline
        64   &   1.977E-03      (2.01)   &   1.595E-01      (0.98)      &   3         (7)\\ \hline
       128  &    5.011E-04      (1.98)  &    8.032E-02      (0.99)     &    4         (7)\\ \hline
       256  &    1.253E-04      (2.00)  &    4.031E-02      (0.99)     &    4         (7)\\ \hline
       512  &    3.134E-05      (2.00)  &    2.019E-02      (1.00)     &    4         (7)\\ \hline
      1024  &    7.828E-06      (2.00)  &    1.010E-02      (1.00)     &    4        (7)\\ \hline
\end{tabular}
\end{center}
\end{table}

\textbf{Example 2}. We test a three-dimensional example, specifically Example 3 from \cite{hou2013weak}.
The initial mesh $\mathcal{T}_0$ is generated by first uniformly partitioning the computational domain $\Omega=(-1,1)^3$ into $5^3$ cubes, followed by subdividing each cube into six tetrahedra. 
The mesh $\mathcal{T}_l$ ($l>1$) is a uniform refinement of $\mathcal{T}_{l-1}$, achieved by subdividing each tetrahedron in $\mathcal{T}_{l-1}$ into eight tetrahedra (see \cite{bey1995tetrahedral}).
Similar numerical results are reported in Table~\ref{ex2_1}.

\begin{table}[H]
\caption{Numerical results of  \textbf{Example 2}.\label{ex2_1}}
\begin{center}
\begin{tabular}{|c|c|c |c|}
\hline
     $l$  &   $\|e_h\|_{L^2}$ (rate)  &    $|e_h|_{H^1} $ (rate)        &      ${\rm Iter_1(Iter_2)}$       \\ \hline
          0  &    9.224E-02       (~---~)  &   8.687E-01    (~---~)    &     2   (--)\\ \hline
         1   &   2.725E-02      (1.76 )  &   4.363E-01      (0.99)    &     2         (7)\\ \hline
         2   &   7.023E-03      (1.96)   &   2.162E-01      (1.01)    &     2         (9)\\ \hline
         3   &   1.755E-03      (2.00)   &   1.076E-01      (1.01)    &     3        (10)\\ \hline
\end{tabular}
\end{center}
\end{table}

\section{Conclusions}\label{sec_con}
In this paper, we developed an IFE method for interface problems with anisotropic coefficients and nonhomogeneous jump conditions based on unfitted meshes. The proposed IFE method maintains the same degrees of freedom as traditional nonconforming FEMs while incorporating the jump conditions into the IFE space to achieve optimal approximation.
We provided a complete theoretical analysis of the method, including the existence and uniqueness of the IFE space, optimal error estimates of the IFE method, and the usual bound on the condition number of the stiffness matrix.
To efficiently solve the resulting linear system, we proposed a preconditioner that employs a Gauss-Seidel smoother with interface correction, ensuring robustness against large jumps in the diffusion coefficients. Numerical experiments verified the optimal convergence of our method and demonstrated the effectiveness of the preconditioner.

\appendix 
\section{Proof of Lemma~\ref{lem_xi_app}}\label{sec_proof}
Let $J_T$ be the $L^2$ projection operator onto $\mathbb{P}_1(\mathscr{B}_T)$. Recalling that $\mathscr{B}_T$ is a ball satisfying $\mathrm{diam}(\mathscr{B}_T)\lesssim h_T$, we have the following well-known properties of $J_T$:
\begin{subequations}\label{esti_pro_L2}
\begin{align}
&|v-J_Tv|_{H^m(\mathscr{B}_T)}\lesssim h_T^{2-m}|v|_{H^2(\mathscr{B}_T)} \quad \forall v\in H^2(\mathscr{B}_T), \quad m=0,1,\label{esti_pro_L2_1}\\
&|v-J_Tv|_{L^\infty(\mathscr{B}_T)}\lesssim h_T^{2-N/2}|v|_{H^2(\mathscr{B}_T)}\quad \forall v\in H^2(\mathscr{B}_T), \label{esti_pro_L2_2}\\
&|J_T v|_{H^1(\mathscr{B}_T)}\lesssim |v|_{H^1(\mathscr{B}_T)} \quad \forall v\in H^1(\mathscr{B}_T).\label{esti_pro_L2_3}
\end{align}
\end{subequations}
To prove (\ref{result_lem_xi}), we add and subtract $J_{T}u_E^\pm$ and use the triangle inequality to get
\begin{equation}\label{pro_int_er_00}
|u_E^\pm-\xi_T^\pm|_{H^m(T)}\leq |u_E^\pm-J_{T}u_E^\pm|_{H^m(T)}+ |J_{T}u_E^\pm-\xi_T^\pm|_{H^m(T)}.
\end{equation}
The first term on the right-hand side can be estimated directly by (\ref{esti_pro_L2_1}), so our main task is to estimate the second term. The proof proceeds in five steps.

\textbf{Step 1 (Decomposition).} Since $(J_{T}u_E^+-\xi_T^+,J_{T}u_E^--\xi_T^-)\in \mathbb{P}_1(\mathbb{R}^N)\times \mathbb{P}_1(\mathbb{R}^N)$, we decompose it as
\begin{equation*}
(J_{T}u_E^+-\xi_T^+,J_{T}u_E^--\xi_T^-)=\sum_{i=0}^{N}(\psi_{i,T}^+,\psi_{i,T}^-)\alpha_{i,T}+\sum_{F\in\mathcal{F}_T}(\phi_{F,T}^+,\phi_{F,T}^-)\gamma_{F,T},
\end{equation*} 
where $\psi_{i,T}^\pm$ and $\phi_{F,T}^\pm$ are defined in (\ref{def_psi_phi1}) and (\ref{def_psi_phi2}) , and $\alpha_{i,T}$ and $\gamma_{F,T}$ are constants defined by
\begin{equation*}
\alpha_{i,T}=\mathcal{J}_{i,T}(J_{T}u_E^+-\xi_T^+,J_{T}u_E^--\xi_T^-),\quad \gamma_{F,T}=\mathcal{M}_{F,T}(J_{T}u_E^+-\xi_T^+,J_{T}u_E^--\xi_T^-).
\end{equation*}
By (\ref{esti_basis}) and (\ref{esti_aux}), we  have the estimate 
\begin{equation}\label{pro_int_er_01}
|J_{T}u_E^\pm-\xi_T^\pm|_{H^m(T)}\lesssim h_T^{N/2-m}\left(\sum_{i=0}^{N-1}|\alpha_{i,T}|+\sum_{F\in\mathcal{F}_T}|\gamma_{F,T}|\right)+h_T^{N/2+1-m}|\alpha_{N,T}|
\end{equation}

\textbf{Step 2 (Estimating $\alpha_{i,T}$, $i=0,\cdots, N-1$).} Using (\ref{def_J}), (\ref{dis_jp_1}) and (\ref{p1.2}), we get
\begin{equation*}
\begin{aligned}
\alpha_{i,T}&=\mathcal{J}_{i,T}(J_{T}u_E^+-\xi_T^+,J_{T}u_E^--\xi_T^-)=\mathcal{J}_{i,T}(J_{T}u_E^+,J_{T}u_E^-)-\mathcal{J}_{i,T}(\xi_T^+,\xi_T^-)\\
&=[\![J_{T}u_E^\pm]\!](\bar{x}_{i,T})-g_D(\tilde{x}_{i,T})=[\![J_{T}u_E^\pm]\!](\bar{x}_{i,T})-[\![u_E^\pm]\!](\tilde{x}_{i,T})\\
&=[\![J_{T}u_E^\pm]\!](\tilde{x}_{i,T})-[\![u_E^\pm]\!](\tilde{x}_{i,T})+\left([\![J_{T}u_E^\pm]\!](\bar{x}_{i,T})-[\![J_{T}u_E^\pm]\!](\tilde{x}_{i,T})\right)\\
&=[\![J_{T}u_E^\pm]\!](\tilde{x}_{i,T})-[\![u_E^\pm]\!](\tilde{x}_{i,T})+ [\![\nabla (J_{T}u_E^\pm)]\!]\cdot (\bar{x}_{i,T}-\tilde{x}_{i,T}).
\end{aligned}
\end{equation*}
Using (\ref{est_x_t_x_star}), (\ref{esti_pro_L2_2}) and (\ref{esti_pro_L2_3}), we obtain for $i=0,\cdots, N-1$ that
\begin{equation*}
\begin{aligned}
|\alpha_{i,T}|&\lesssim \sum_{s=\pm}\|J_{T}u_E^s-u_E^s\|_{L^\infty(\mathscr{B}_T)}+h_T^{2-N/2}\sum_{s=\pm}|J_Tu_E^s|_{H^1(\mathscr{B}_T)}\\
& \lesssim h_T^{2-N/2}\sum_{s=\pm}\|u_E^s\|_{H^2(\mathscr{B}_T)}.
\end{aligned}
\end{equation*}

\textbf{Step 3 (Estimating $\alpha_{N,T}$).}
It follows from  (\ref{def_J}), (\ref{dis_jp_1})  and (\ref{p1.3}) that
\begin{equation*}
\begin{aligned}
|\alpha_{N,T}|&=\left|[\![\mathbb{B}_T^\pm \nabla (J_T u_E^\pm) \cdot \bar{n}_T]\!]-\mathrm{avg}_{\Gamma_{T}^{ext}}([\mathbb{B}\nabla u\cdot n]_\Gamma)\right|\\
&=\left|\mathrm{avg}_{\Gamma_{T}^{ext}}\left([\![\mathbb{B}_T^\pm \nabla (J_T u_E^\pm)\cdot \bar{n}_T]\!]-[\mathbb{B}\nabla u\cdot n]_\Gamma\right)\right|\\
&\lesssim \sum_{s=\pm}\mathrm{avg}_{\Gamma_{T}^{ext}}\left(\left|\bar{n}_T^\top\mathbb{B}_T^s \nabla (J_T u_E^s)-n^\top\mathbb{B}^s\nabla u^s_E\right|\right),
\end{aligned}
\end{equation*}
where we used the fact that $[\![\mathbb{B}_T^\pm \nabla (J_T u_E^\pm) \cdot \bar{n}_T]\!]$ is a constant in the second identity. 
By (\ref{esti_B}) we deduce that for all $x\in \Gamma_{T}^{ext}$,
\begin{equation*}
\begin{aligned}
&\left|\bar{n}_T^\top\mathbb{B}_T^s \nabla (J_T u_E^s)-(n^\top\mathbb{B}^s\nabla u^s_E)(x)\right|\\
&\qquad\leq \left|\bar{n}_T^\top\mathbb{B}_T^s \nabla (J_T u_E^s)-\bar{n}_T^\top\mathbb{B}_T^s\nabla u^s_E\right|+ \left|(\bar{n}_T^\top\mathbb{B}_T^s-n^\top\mathbb{B}^s)\nabla u^s_E\right|\\
&\qquad\lesssim \left| \nabla (J_T u_E^s)-\nabla u^s_E\right|+ h_T\left|\nabla u^s_E\right|.
\end{aligned}
\end{equation*}
Combining the above results and using (\ref{trace_avg}) and (\ref{esti_pro_L2_1}) yields
\begin{equation*}
\begin{aligned}
|\alpha_{N,T}|&\lesssim  \sum_{s=\pm}\left(\mathrm{avg}_{\Gamma_{T}^{ext}}(\left| \nabla (J_T u_E^s)-\nabla u^s_E\right|)+h_T\mathrm{avg}_{\Gamma_{T}^{ext}}(\left|\nabla u^s_E\right|)\right)\\
&\lesssim \sum_{s=\pm}\left(h_T^{-N/2}\| \nabla (J_T u_E^s)-\nabla u^s_E\|_{L^2(\mathscr{B}_T)}+h_T^{1-N/2}(| u^s_E|_{H^2(\mathscr{B}_T)}+| u^s_E|_{H^1(\mathscr{B}_T)})\right)\\
&\lesssim h_T^{1-N/2}\sum_{s=\pm}\|u_E^s\|_{H^2(\mathscr{B}_T)}
\end{aligned}
\end{equation*}

\textbf{Step 4 (Estimating $\gamma_{F,T}$).} 
By (\ref{def_MFT}) and (\ref{dis_jp_2}), we have for all $F\in\mathcal{F}_T$ that 
\begin{equation*}
\begin{aligned}
\gamma_{F,T}&=\mathcal{M}_{F,T}(J_{T}u_E^+-\xi_T^+,J_{T}u_E^--\xi_T^-)=\mathcal{M}_{F,T}(J_{T}u_E^+-u_E^+,J_{T}u_E^--u_E^-)\\
&= |F|^{-1}\sum_{s=\pm}\int_{F\cap\partial T_h^s}J_{T}u_E^s-u_E^s.
\end{aligned}
\end{equation*}
By H\"older's inequality, the trace inequality and (\ref{esti_pro_L2_1}), we deduce that 
\begin{equation*}
\begin{aligned}
|\gamma_{F,T}|&\leq |F|^{-1}\sum_{s=\pm}\int_{F\cap\partial T_h^s}|J_{T}u_E^s-u_E^s|\leq |F|^{-1}\sum_{s=\pm}\int_{F}|J_{T}u_E^s-u_E^s|\\
&\leq |F|^{-1/2}\sum_{s=\pm}\|J_{T}u_E^s-u_E^s\|_{L^2(F)}\\
&\lesssim  |F|^{-1/2}\sum_{s=\pm}\left(h_T^{-1/2}\|J_{T}u_E^s-u_E^s\|_{L^2(T)}+h_T^{1/2}|J_{T}u_E^s-u_E^s|_{H^1(T)}\right)\\
& \lesssim h_T^{2-N/2}\sum_{s=\pm}\|u_E^s\|_{H^2(\mathscr{B}_T)},
\end{aligned}
\end{equation*}
where in the last inequality we also used  $|F| \gtrsim h_T^{(N-1)}$ for all $F\in\mathcal{F}_T$ since the  triangulation is  shape regular.

\textbf{Step 5.} 
Finally, substituting the estimates of $\alpha_{i,T}$ and $\gamma_{F,T}$ in above steps into (\ref{pro_int_er_01}) we get 
\begin{equation*}
|J_{T}u_E^\pm-\xi_T^\pm|_{H^m(T)}\lesssim h_T^{2-m}\sum_{s=\pm}\|u_E^s\|_{H^2(\mathscr{B}_T)},
\end{equation*}
which together with (\ref{pro_int_er_00}) and (\ref{esti_pro_L2_1}) implies the desired result (\ref{lem_xi_app}).


\bibliographystyle{plain}
\bibliography{refer}
\end{document}